\documentclass[11pt,a4paper,twoside]{article}
\usepackage{amsthm,amssymb,amsfonts,amsmath}

\newtheorem{theorem}{Theorem}
\newtheorem{example}{Example}%
\newtheorem{corollary}{Corollary}%
\newtheorem{Proposition}{Proposition}%
\newtheorem{Lemma}{Lemma}%

\newtheorem{definition}{Definition}%

\def\Ric{\operatorname{Ric}}
\newcommand\oc{\overset{\centerdot}}

\title{Characterization of Sasakian Manifolds by Curvature Conditions}

\author{
     Sourav Nayak\footnote{Department of Mathematics, Indian Institute of Technology - Hyderabad, Sangareddy-502285, India
     \newline e-mail: {\tt ma22resch11004@iith.ac.in},
     Orcid: 0009-0003-4330-8283}\ ,
     Dhriti Sundar Patra\footnote{Department of Mathematics, Indian Institute of Technology - Hyderabad, Sangareddy-502285, India
     \newline e-mail: {\tt dhriti@math.iith.ac.in} \ and \ {\tt dhritimath@gmail.com} }
        \ and
     Vladimir Rovenski \footnote{Department of Mathematics, University of Haifa, Mount Carmel, 3498838 Haifa, Israel
     \newline e-mail: {\tt vrovenski@univ.haifa.ac.il},
     Orcid: 0000-0003-0591-8307
     }
}

\date{}

\begin{document}

\maketitle

\begin{abstract}
Sasakian manifolds, a subclass of contact metric manifolds, play a central role in modern differential geometry and mathematical physics, where they arise naturally in geometric mechanics, CR geometry, and supersymmetric field theories.
The weak contact metric (w.c.m.) structure
generalizes the contact metric structure and provides a broader framework for studying contact geometry and its applications.
A fundamental problem is to understand how curvature conditions constrain the underlying contact metric structure and, in particular, how they distinguish Sasakian geometry from its weaker variants.
Using the partial Ricci flow,
we
characterize Sasakian structure (among w.c.m.~manifolds) under 
the $(\kappa,\mu)$-nullity condition related to curvature.
For $\kappa<1$ we find conditions under which a w.c.m.~mani\-fold admits a bi-Legendrian structure, and for $\kappa=\mu=0$ establish splitting and classification~results.


\end{abstract}

\maketitle

\section{Introduction and Main Results} \label{sec:intro}

Contact Riemannian geometry is a significant field in mathematics and mathematical physics~\cite{blair2010riemannian}.
It considers a smooth manifold $M$ of dimension $2n + 1$ equipped with a contact metric structure $(\varphi,\xi,\eta,g)$.
Here, $g$ is a Riemannian metric,
$\varphi$ is a skew-symmetric (1,1)-tensor of rank $2n$, $\xi$ is a unit (Reeb) vector field, and $\eta$ is a 1-form
such that $\eta(\xi) = 1$. These tensors satisfy the~relations
\[
 \varphi^2=-I+\eta\otimes\xi,\qquad
 g(\varphi X, \varphi Y) = g(X, Y) - \eta(X)\eta(Y),\qquad
 d\eta(X,Y)=g(X,\varphi Y),
\]
for all $X, Y \in \mathfrak{X}_M$, where
$I$ is the identity operator and $\mathfrak{X}_M$ is the space of smooth vector fields on~$M$.
In this case, $\eta$ is a contact 1-form:
$\eta\wedge(d\eta)^n\ne0$,
and
the $\xi$-trajectories are geodesics.
A~{$K$}-contact manifold is a special case of a contact metric manifold, whose Reeb vector field ${\xi}$ is Killing (an infinitesimal isometry).
{Sasakian manifolds} (a subclass of $K$-contact manifolds)
occupy a central position at the crossroads of Riemannian, contact, and complex geometry,
and are characterized
among contact metric manifolds by
the following
equality, e.g., Theorem~6.3
in~\cite{blair2010riemannian}:
\begin{align}\label{E-1k-0-nullity}
R_{X,Y}\,\xi &=
\eta(X)\varphi^2Y -\eta(Y)\varphi^2X ,
\end{align}
where
$R_{X,Y}=[\nabla_X, \nabla_Y] -\nabla_{[X,Y]}$ is the curvature tensor
and $\nabla$ is the Levi-Civita connection.
Neither Sasakian nor $K$-contact manifolds satisfy the following condition:
\begin{align}\label{E-RXY-xi}
 R_{X,Y}\,\xi = 0 ,
\end{align}
-- $\xi$ belongs to the nullity distribution.
In fact, any $(2n+1)$-dimensional contact metric manifold with $n>1$ and \eqref{E-RXY-xi} is locally isometric to the product of a flat Riemannian
space and a round $n$-sphere, see \cite[Theorem~7.5]{blair2010riemannian};
thus, for $n=1$, such manifolds are flat.
Generalizing \eqref{E-1k-0-nullity} and \eqref{E-RXY-xi}, $(\kappa,\mu)$-manifolds (the  centerpiece of our paper) were introduced as those contact metric manifolds for which $\xi$ belongs to the $(\kappa,\mu)$-nullity distribution, where $\kappa,\mu\in\mathbb{R}$ and
$h=\tfrac12\,\mathcal{L}_{\xi}\,\varphi$ given in~\eqref{Eq-4.1}:
\begin{align}\label{E-k-mu-nullity}
 R_{X,Y}\,\xi = \kappa\{\eta(X)\varphi^2Y
 -\eta(Y)\varphi^2X\}+\mu\{\eta(Y) h X-\eta(X) h Y\},
\end{align}
see \cite[Section~7.3]{blair2010riemannian}.
Koufogiorgos et.~al
\cite{K2008} defined $(\kappa,\mu,\nu)$-manifolds and showed~that for dimensions $>3$ such a manifold is a $(\kappa,\mu)$-manifold.
The motivation for studying these manifolds is that they naturally generalize Sasakian manifolds, admit a bi-Legendrian structure, and in the non-Sasakian case the $(\kappa,\mu)$-nullity condition determines the curvature completely, e.g., \cite{Boeckx2000,GF2021,CMdiT-2008}.

\smallskip

Rovenski et al.
\cite{rov-119,rov-RS23,RWo-2} presented weak metric structures on manifolds that extend the concept of almost Hermitian structure and almost contact metric structure. The weak contact metric (w.c.m.) structure, where the complex structure on the contact distribution is replaced by a non-degenerate skew-symmetric tensor, see Definition~\ref{D-02}, allowed for a fresh perspective on traditional theory and the discovery of applications in geometry and physics.
For example, weak $K$-contact manifolds have positive $\xi$-sectional curvature; hence, they form a larger class than $K$-contact manifolds, for which this curvature is equal to 1.
In \cite{rov-RS23}, conditions are given for a weak $K$-contact manifold equipped with a generalized Ricci soliton structure to be an Einstein~manifold.
In~\cite{rov-119}, it was shown that the Sasakian structure is rigid in a sense:
a w.c.m. manifold is weak Sasakian if and only if it is a Sasakian manifold.
Advances in geometry and modern physics
(such as
non-commutative geometry and superspace, string theory and gravity, and the asymmetric energy-momentum tensor)
have made the non-symmetric metric tensor
$G=g+F$, where $g$ is a pseudo-Riemannian metric (associated with gravity), and $F\ne0$ is a skew-symmetric tensor (associated with electromagnetism), more attractive than ever.
Weak metric structures are fully consistent with the skew-symmetric part
$F(X,Y)=g(X,\varphi Y)$
of the tensor $G$; and provide new tools for studying
non-symmetric Riemannian manifolds $(M,G)$, see~\cite{RZ-1}.

\smallskip

In the paper, we study the questions:
\textit{Which weak $(\kappa,\mu)$-manifolds
admit Sasakian structure or have conjugate Legendrian foliations}?
\textit{Which w.c.m. manifolds with
the curvature-related property
\eqref{E-RXY-xi} split}?
The aim of this paper is to provide a clear curvature-based mechanism that distinguishes Sasakian geometry within the broader class of w.c.m. structures.
Our analysis
is carried out under the assumptions, which are trivially satisfied by contact metric~manifolds:
\begin{align}\label{Eq-2assump}
 \mathcal{L}_{\xi}\,{Q}&=0 , \\
\label{E-nS-10}
 ((\nabla_X\,Q)Y)^\top&=0\qquad (X,Y\in\ker\eta) ,
\end{align}
where $^\top$ denotes the $\ker\eta$\,-component of a vector,
and $\mathcal{L}$ is the Lie derivative.
 The following theorem shows that a w.c.m. manifold admits a Sasakian structure obtained as the 
limit under the partial Ricci flow, see Section~\ref{sec:02.2}, of w.c.m.\,structures
with $\ker\eta$-adapted metrics satisfying \eqref{E-1k-0-nullity}.

\begin{theorem}\label{P-K-S}
Let a w.c.m.\, manifold $M^{2n+1}(\varphi_0,Q_0,\xi,\eta,g_0)$ with 
the
property \eqref{E-1k-0-nullity}
satisfy the condition \eqref{Eq-2assump}.
Then $(\varphi_0,Q_0,\xi,\eta,g_0)$ is a weak $K$-contact structure with $\mathcal{N}^{(5)}=0$, and there exists a smooth family $(\varphi_t,Q_{\,t},\xi,\eta,g_t)\ (t\in\mathbb{R})$
of weak $K$-contact structures on $M$ with  $\ker\eta$-adapted metrics $g_t$ (i.e., $g_t$ changes on $\ker\eta$~only) satisfying  \eqref{E-1k-0-nullity}, that converges exponentially fast, as $t\to-\infty$, to a Sasakian structure.
\end{theorem}

The following result
generalizes \cite[Theorem~7.7]{blair2010riemannian}
and dynamically characterizes Sasakian mani\-folds among w.c.m. manifolds with
the $(1,\mu)$-nullity condition and \eqref{Eq-2assump},
and for $\kappa<1$ we find a bi-Legendrian structure,
see Definition~\ref{D-002}, on a w.c.m. manifold satisfying the condition \eqref{E-nS-10}.

\begin{theorem}\label{T-04}
Let a w.c.m. manifold
$M^{2n+1}(\varphi_0,Q_0,\xi,\eta,g_0)$ satisfy
\eqref{Eq-2assump} and
the $(\kappa,\mu)$-nullity condition \eqref{E-k-mu-nullity}.
Then $\kappa\le 1$, and

\smallskip
\noindent\
$(i)$ if $\kappa=1$, then the structure is  weak $K$-contact with $\mathcal{N}^{(5)}=0$,
and there exists a smooth family $(\varphi_t,Q_{\,t},\xi,\eta,g_t)\ (t\in\mathbb{R})$
of weak $K$-contact structures on $M$
with $\ker\eta$-adapted metrics $g_t$
satisfying  \eqref{E-1k-0-nullity}
that converges exponentially fast, as $t\to-\infty$, to a Sasakian structure.

\smallskip\noindent\
$(ii)$ if $\kappa<1$ and \eqref{E-nS-10} is true, then the eigenvalues of
$\widetilde h$, see \eqref{Eq-4.1}, are $\{\,0,\pm\sqrt{1-\kappa}\,\}$,
$\ker\widetilde h=\{\xi\}$,
and the orthogonal decomposition
$\,\ker\eta={\cal D}_+\oplus{\cal D}_-$ is true, where ${\cal D}_\pm$ are $n$-dimensional
eigen-distributions of the eigenvalues $\pm\sqrt{1-\kappa}$, respectively, of $\widetilde h$, and the distributions ${\cal D}_\pm$ determine totally geodesic~foliations.
In~other words, ${\cal D}_+$ and ${\cal D}_-$ are conjugate Legendrian foliations.
\end{theorem}

According to the following theorem, if we  apply the transformation \eqref{E-cont-lambda} to a contact metric structure $(\varphi',\xi,\eta,g')$ satisfying the curvature-related condition \eqref{E-RXY-xi}, we obtain a w.c.m. structure $(\varphi,Q,\xi,\eta,g)$ satisfying the $(\kappa',\mu')$-nullity condition
with $\kappa'=3\,{\lambda}^2(1-{\lambda})$
and $\mu'=2\,{\lambda}(1-{\lambda})$.

\begin{theorem}\label{thm6}
$(i)$ Let $M^{2n+1}(\varphi',\xi,\eta,g')$ be a contact metric manifold
and ${\lambda}>0$ be a differentiable function on~$M$. Then $(\varphi,Q,\xi,\eta,g)$ is a w.c.m. structure on $M$, where
\begin{align}\label{E-cont-lambda}
 \varphi ={\lambda}\,\varphi',\qquad
 g ={\lambda}^{-1}g'
 +(1-{\lambda}^{-1})\,\eta\otimes\eta,
 \qquad
 Q ={\lambda}^2 I + (1-{\lambda}^2)\,\eta\otimes\xi.
\end{align}
$(ii)$
Moreover, let
${\lambda}>0$ be real, Then, $(\varphi,Q,\xi,\eta,g)$ satisfies the $(\kappa,\mu)$--nullity condition \eqref{E-k-mu-nullity}
if and only if $(\varphi',\xi,\eta,g')$ satisfies the $(\kappa',\mu')$-nullity condition
with $\kappa'= {\lambda}^2(\kappa+3-{3\lambda})$ and $\mu'=\lambda(\mu+2-{2\lambda})$.
\end{theorem}

The following consequence of Theorems~\ref{T-04} and \ref{thm6} generalizes \cite[Theorem~3]{blair1994}
and characterizes three-dimensional w.c.m. manifolds locally isometric to Lie groups with left-invariant metrics.

\begin{corollary}\label{Thm-lie}
A w.c.m. manifold
$M^{3}(\varphi,Q,\xi,\eta,g)$, satisfying conditions \eqref{E-nS-10} and
\eqref{E-k-mu-nullity} with $\kappa<1$,
is locally isometric to one of the following unimodular Lie groups with a left-invariant metric:


$\bullet$
$SO(3)$ -- rotations about the origin of $\mathbb{R}^3$;

$\bullet$ $SL(2,\mathbb{R})$ -- linear transformations of $\mathbb{R}^2$ that preserve oriented area;

$\bullet$ $E(2)$ -- rigid motions of $\mathbb{R}^2$;

$\bullet$ $E(1,1)$ -- rigid motions of the Minkowski~plane.
\end{corollary}

The following statement, generalizing \cite[Theorem~7.5]{blair2010riemannian}, shows that a $(2n+1)$-dimensional w.c.m. manifold splits under the assumptions \eqref{Eq-2assump} and \eqref{E-nS-10} and the curvature-related condition \eqref{E-RXY-xi}, where one of the factors is equal to
$\mathbb{R}^{n+1}$, and the curvature tensor of the second factor is completely~defined.

\begin{theorem}\label{T-04b}
Let a w.c.m. manifold
$M^{2n+1}(\varphi,Q,\xi,\eta,g)$
with the curvature-related property
\eqref{E-RXY-xi} satisfy conditions
\eqref{Eq-2assump} and \eqref{E-nS-10}.
Then the manifold  splits
along the distributions
\({\cal D}_- \oplus \{\xi\}\)
with flat integral manifolds, and
\({\cal D}_+\), where ${\cal D}_\pm$
are the $n$-dimensional eigen-distributions of the eigenvalues $\pm1$,
respectively, of $\widetilde h$, see \eqref{Eq-4.1},
and $({\cal D}_+,{\cal D}_-)$ is a bi-Legendrian structure.
Moreover, if \(n>1\), then \(Q\,|_{\ker\eta}=\lambda I_{\ker\eta}\)\,
for some constant
\(\lambda>0\), and integral manifolds of \(\mathcal{D}_{+}\) have
constant sectional curvature $1/\sqrt{4\lambda}$.
\end{theorem}

By a transformation analogous to \eqref{E-cont-lambda}, the w.c.m. structure
in Theorem~\ref{T-04b} reduces to the contact metric structure.
The next consequence of Theorem~\ref{T-04b} generalizes a result by Blair~(1976).

{\begin{corollary}\label{cor4.1}
A w.c.m. manifold of dimension $\ge 5$ with conditions \eqref{Eq-2assump} and \eqref{E-nS-10} cannot be flat.
\end{corollary}

The structure of the paper is as follows.  Section~\ref{sec:intro} presents an overview of the study and the main results.
In Section~\ref{sec:1} we review the necessary identities and known results on w.c.m. manifolds, and in Section~\ref{sec:02.2}  we prove that the property \eqref{E-1k-0-nullity}
remains invariant under the partial Ricci flow.
In~Section~\ref{sec4} we prove the main results and give some illustrative examples.
Section~\ref{sec:3} presents auxiliary
statements
concerning the curvature and structure tensors of w.c.m. manifolds.
Our proofs rely on both classical techniques and the properties of newly introduced structure~tensors.

\section{Preliminaries on w.c.m. manifolds}
\label{sec:1}

Here, we review the necessary definitions and results on w.c.m. manifolds; see \cite{rov-119,rov-RS23,Rov-Wa-2021,RWo-2}.

\begin{definition}\label{D-02}
\rm
A \textit{weak almost contact metric structure} on a differentiable connected manifold $M^{2n+1}$ is a set $(\varphi,Q,\xi$, $\eta,g)$, where $g$ is a Riemannian metric on $M$, $\varphi$ is a skew-symmetric (1,1)-tensor of rank $2n$, $\xi$ is a unit (Reeb) vector field, $\eta$ is a 1-form such that ${\eta}({\xi})=1$, satisfying
\begin{align}\label{2.2}
 & g({\varphi} X,{\varphi} Y)= g(X,Q\,Y)
 -{\eta}(X)\,{\eta}(Y)
 \qquad (X,Y\in\mathfrak{X}_M),
\end{align}
and $Q:=-{\varphi}^2 + {\eta}\otimes {\xi}$ is a self-adjoint (1,1)-tensor field
(and $Q=I$ for almost contact metric manifolds).
A~\textit{weak contact metric $($w.c.m.$)$ structure} is a weak almost contact metric structure satisfying
\begin{align}\label{2.2b}
 d\eta(X,Y)=g(X,\varphi Y).
\end{align}
A w.c.m. structure is called a \textit{weak $K$-contact structure} if the
vector field ${\xi}$ is Killing,~i.e., ${\cal L}_{\xi}\,g=0$.
\end{definition}

By \eqref{2.2b},
$\eta$ is a contact 1-form:
$\eta \wedge (d\eta)^n\ne0$, i.e., $(M^{2n+1},\eta)$ is a contact manifold.
The~self-adjoint tensor
$\widetilde Q:=Q-I$
measures the difference between w.c.m. and contact metric~structures.

\begin{definition}\label{D-002}\rm
An $n$-dimensional distribution $D$ on a contact manifold $(M^{2n+1},\eta)$ is called a \textit{Legendrian distribution} if $d\eta(X,Y)=0$ for all vectors $X,Y$ in $D$. When the distribution $D$ is involutive, it defines a \textit{Legendrian foliation} on $(M^{2n+1},\eta)$.
Two Legendrian distributions $D$ and $D'$, or, foliations ${\cal F}$ and ${\cal F}'$,
of a w.c.m. manifold $M^{2n+1}(\varphi,Q,\xi,\eta,g)$
are said to be \textit{conjugate} if $\varphi(D')=D$,
or $\varphi(T{\cal F}')=T{\cal F}$, respectively,
and the tangent bundle splits as the orthogonal sum
$TM=D\oplus D'\oplus\{\xi\}$. Conjugate distributions or foliations form a pair known as a \textit{bi-Legendrian structure} on $M^{2n+1}$.
\end{definition}

The
tensor $Q$ of a w.c.m. structure $({\varphi},Q,\xi,\eta,g)$ is positive definite, and the following is true:
\begin{align}\label{P-6}
 {\varphi}\,\xi=0,\quad
 \eta\circ{\varphi}=0,\quad
 [Q,\,{\varphi}]=0,\quad
 Q\,\xi=\xi,\quad
 \widetilde{Q}\,\xi=0,\quad
 \eta\circ Q=\eta,\quad
 \eta\circ\widetilde{Q}=0,
\end{align}
the $\xi$-trajectories are geodesics: $\nabla_\xi\,\xi=0$, a $2n$-dimensional distribution $\ker\eta$ is $\varphi$-invariant,
but it is not involutive:
\begin{align*}
 g([X, {\varphi} X], {\xi})= 2\,d{\eta}({\varphi} X,X) = 2\,g({\varphi} X,{\varphi} X)>0\quad (X\in\ker\eta\setminus\{0\}).
\end{align*}

The~{fundamental 2-form} ${F}(X,Y)=g(X,\varphi Y)$
of a w.c.m. manifold
is closed: $d{F}
=0$; therefore,
\begin{align}\label{dPhi}
 (\nabla_X {F})(Y,Z)+(\nabla_Y {F})(Z,X)+(\nabla_Z {F})(X,Y)= 0 \quad (X,Y,Z\in\mathfrak{X}_M) .
\end{align}
The covariant derivative of
${F}(Y,Z)=g(Y,\varphi Z)$ along $X\in\mathfrak{X}_M$ gives
\begin{align}\label{E-nabla-phi-Phi}
 (\nabla_X{F})(Y,Z)=g(Y,(\nabla_X\varphi)Z). 
\end{align}
The Lie derivatives of $\varphi$ and
$g$ in the $Z$-direction are given by
\begin{align}\label{E-Killing}
 (\mathcal{L}_Z\,\varphi)X
 = [Z, \varphi X] - \varphi\,[Z, X] , \qquad
 (\mathcal{L}_{Z}\,g)(X,Y)
  = g(\nabla_Y {Z}, X) + g(\nabla_X {Z}, Y) .
\end{align}

\begin{Proposition}[\cite{rov-119}]\label{thm6.2}
For a w.c.m. structure $(\varphi,Q,\xi,\eta,g)$, we have
$d\eta(\varphi X, Y) = d\eta(\varphi Y, X)$ and
\begin{align}\label{3.1AA}
 2\,g((\nabla_{\xi}\,\varphi)X, Y)
 = g([\xi,\varphi Y],\widetilde QX)
 -g([\xi,\varphi X],\widetilde QY)
 -\xi(g(\widetilde QX, \varphi Y)),
\end{align}
in addition, the tensor $\mathcal{L}_{\xi}\,\varphi$
vanishes if and only if $\,\xi$ is a Killing vector field, i.e., $\mathcal{L}_{\xi}\,g=0$.
\end{Proposition}

For the sake of abbreviation, we have set
\begin{align}\label{Eq-4.1}
 h = \tfrac{1}{2}\,\mathcal{L}_{\xi}\,\varphi,\qquad
 \widetilde h:=Q^{-1} h.
\end{align}

Using $X=\xi$ in \eqref{3.1AA}, gives $(\nabla_\xi\,\varphi)\,\xi=0$.
By \eqref{E-Killing}, we obtain
\begin{align*}
 2\,h X
 = (\nabla_{\xi}\,\varphi)X
 +\varphi\nabla_X\,\xi -\nabla_{\varphi X}\,\xi.
\end{align*}
Hence, $h\,\xi=0$ is true, and the distribution $\ker\eta$ of a w.c.m. manifold is $h$-invariant: $2\,g(hX,\xi)=0$ for $X\in\ker\eta$.
Note that $\widetilde h\,\xi=0$ and~$\widetilde h\circ\eta=0$.
The~following lemma
generalizes the properties of a contact metric manifold such that the tensor $h$ is self-adjoint and anti-commutes with~$\varphi$.

\begin{Proposition}\label{C-2cond}
Let a w.c.m. manifold satisfy the condition \eqref{Eq-2assump}, then $h$ is a self-adjoint traceless $(1,1)$-tensor that anti-commutes with $\varphi$ and commutes with $Q$; in addition,
$\nabla_{\xi}\,{Q}=\nabla_{\xi}\,\varphi=0$,
${\rm trace}\,h=0$,
and
\begin{equation}\label{E-xi-n}
 \nabla\,\xi = -\varphi - \varphi\,\widetilde h.
\end{equation}
A w.c.m. manifold
is weak K-contact if and only if, see \cite{rov-RS23},
\begin{equation}\label{E-xi-nK}
 \nabla\,\xi = -\varphi .
\end{equation}
\end{Proposition}

\begin{proof}
 On a w.c.m. manifold, the tensor $h$ and its conjugate $(1,1)$-tensor $h^*$ satisfy, see \cite{rov-119},
\begin{align}
\label{E-31}
 h-h^* & =  \nabla_{\xi}\,\varphi ,\\
 \label{E-31A}
 h\varphi+\varphi h &= -\tfrac12 \mathcal{L}_\xi\,Q
 ,\\
 \label{E-31B}
 h Q - Q h &= \tfrac12\,[\varphi, \mathcal{L}_\xi\,Q], \\
\label{E-30}
 Q\,\nabla \xi &= -\varphi(Q + h^*).
\end{align}
By conditions, \eqref{E-31A} and \eqref{E-31B}, $h$ anti-commutes with $\varphi$ and commutes with $Q$. The~same is true for $h^*$.
Since $Q$ is a self-adjoint positive definite (1,1)-tensor, the tensor $h$ is also self-adjoint, that is, $h=h^*$. Hence, \eqref{E-31} gives
$\nabla_{\xi}\,\varphi=0$.
By \eqref{E-30} and \eqref{P-6}, we have
 $\nabla \xi= - \varphi - \varphi\, Q^{-1} h^*$.
From this, $h=h^*$ we obtain \eqref{E-xi-n}.
Using \eqref{E-xi-n}, we get $\nabla_{\xi}\,{Q}=0$:
\begin{align*}
 0 = (\mathcal{L}_{\xi}\,{Q})X
 = [\xi, QX] - Q[\xi, X]
 = (\nabla_{\xi}\,{Q})X +Q\nabla_X\,\xi
 -\nabla_{QX}\,\xi
  = (\nabla_{\xi}\,{Q})X.
\end{align*}
If $h X=\lambda X$, then using $h\varphi = - \varphi h$
(by assumptions and \eqref{E-31A}),
we get $h \varphi X=-\lambda \varphi X$.
Thus, if $\lambda$ is an eigenvalue of $h$, then $-\lambda$ is also an eigenvalue of $h$; hence, ${\rm trace}\,h=0$.
By the above and Proposition~\ref{thm6.2},
a w.c.m. manifold is weak K-contact if and only if \eqref{E-xi-nK} is true.
\end{proof}

Since $\nabla\,I=0$, we get
$\nabla\,\widetilde Q =\nabla\,Q$
and $\mathcal{L}_{\xi}\,\widetilde{Q}=\mathcal{L}_{\xi}\,{Q}$.
By this, Proposition~\ref{C-2cond}  and
\eqref{Eq-2assump}-\eqref{E-nS-10}, we find
\begin{align}\label{E-nS-10b}
 (\nabla_X Q)Y
= -\eta(Y)\widetilde Q\nabla_X\,\xi
 - g(\widetilde Q\nabla_X\,\xi,Y) \,\xi .
\end{align}


The Nijenhuis torsion of $\varphi$ is given by
 $N_\varphi(X,Y) =
 {\varphi}^2 [X,Y] + [{\varphi} X, {\varphi} Y] - {\varphi}[{\varphi} X,Y] - {\varphi}[X,{\varphi} Y]$.

\begin{Lemma}[see Corollary~2.1 in \cite{rov-119}]
For a w.c.m.\,structure, we have
\begin{align}\label{Eq-nabla-phi-N5}
 2\,g((\nabla_X \varphi )Y,Z) =g(\mathcal{N}^{(1)}(Y,Z),\varphi X)
 + 2\,g(\varphi X,\varphi Y)\,\eta(Z)
 -2\,g(\varphi X,\varphi Z)\,\eta(Y) +\mathcal{N}^{(5)}(X,Y,Z) ,
\end{align}
where
$\mathcal{N}^{(1)}:=N_\varphi+2\,d\eta\otimes\xi$, and the $(0,3)$-tensor ${\cal N}^{\,(5)}$ is given by
\begin{align*}
\notag
 {\cal N}^{\,(5)}(X,Y,Z) =&\,\varphi Z(g(X, \widetilde QY))
 -\varphi Y(g(X, \widetilde QZ))
 -X(g(\varphi Z, \widetilde QY))
 +g([X, \varphi Z], \widetilde QY)\\
 & -\,g([X,\varphi Y], \widetilde QZ) +g([Y,\varphi Z] -[Z, \varphi Y] - \varphi[Y,Z],\ \widetilde Q X).
\end{align*}
\end{Lemma}

The new tensor $\mathcal{N}^{\,(5)}$ complements the series of well-known tensors $\mathcal{N}^{\,(i)}\ (i=1,2,3,4)$, see~\cite{blair2010riemannian}.

We derive
the covariant derivative of the tensors $\varphi$ and $Q$ for a w.c.m.~manifold satisfying
\eqref{Eq-2assump}.

\begin{Proposition}\label{L-cond-equiv}
For a w.c.m.~manifold satisfying \eqref{Eq-2assump}, the property \eqref{E-1k-0-nullity} is equivalent to the following:
\begin{align}\label{E-nabla=phi-nullity}
 (\nabla_X\,\varphi)Y & = \eta(Y)\varphi^2X -g(X,\varphi^2Y)\xi .
\end{align}
If any of these conditions holds, then the structure is weak $K$-contact and
the equality is true,
\begin{align}\label{E-nS-10K}
 (\nabla_X Q)Y = \eta(Y)\widetilde Q \varphi X
 + g(\widetilde Q \varphi X, Y)\,\xi .
 \end{align}
\end{Proposition}

\begin{proof}
 \eqref{E-1k-0-nullity} $\Rightarrow$
\eqref{E-nabla=phi-nullity}.
For $Y\perp\xi$ and $X=\xi$, \eqref{E-1k-0-nullity} reads $R_{\,\xi,Y}\,\xi = -QY$.
By this and \eqref{E-37}, we get $h^2 Y=0$. Since $h$ is a self-adjoint (1,1)-tensor, $h=0$ is true, and by Proposition~\ref{thm6.2}, $\xi$ is a Killing vector field.
Therefore, $(\varphi,Q,\xi,\eta,g)$
is a weak $K$-contact structure, and by \cite[Proposition~2]{rov-RS23}, we~get
\begin{align}\label{E-R-varphi}
 R_{\,\xi,X}=\nabla_X\,\varphi.
\end{align}
From \eqref{E-1k-0-nullity}, \eqref{E-R-varphi} and the first Bianchi identity, we obtain the equality
\[
 (\nabla_X\varphi)Y-(\nabla_Y\varphi)X
 =\eta(Y)\varphi^2X - \eta(X)\varphi^2\,Y .
\]
By~\eqref{dPhi}, since $d{F}=0$, the equation above yields
\[
 g((\nabla_Z \varphi)Y, X) =
 g((\nabla_X \varphi)Y - (\nabla_Y \varphi)X, Z)
 =g( \eta(Y) \varphi^2X-\eta(X) \varphi^2\,Y, Z) ,
\]
which reduces~to \eqref{E-nabla=phi-nullity}.

\smallskip

\eqref{E-nabla=phi-nullity} $\Rightarrow$
\eqref{E-1k-0-nullity}.
The covariant derivative of $\varphi\,\xi=0$, using \eqref{E-xi-n},
\eqref{E-nabla=phi-nullity}, and the definition of $Q$, gives $\widetilde hX=0$, hence, $h = 0$. By Proposition~\ref{thm6.2}, $\xi$ is a Killing vector field,
and $(\varphi,Q,\xi,\eta,g)$
is a weak $K$-contact structure.
Therefore, we have \eqref{E-R-varphi}.
Applying this in conjunction with
\eqref{E-nabla=phi-nullity}, we have
\[
 R_{\,\xi,X}Y= \eta(Y)\varphi^2X -g(\varphi^2X,Y)\xi.
 \]
From this, using the first Bianchi identity,
we get~\eqref{E-1k-0-nullity}.

Let \eqref{E-nabla=phi-nullity} be true. Then, $(\varphi,Q,\xi,\eta,g)$ is a weak $K$-contact structure. By Proposition~\ref{C-2cond}, \eqref{E-xi-n} holds with $h=0$.
Taking the covariant derivative of $Q=-{\varphi}^2 + {\eta}\otimes {\xi}$ and using \eqref{E-nabla=phi-nullity} along with $\nabla\xi =-\varphi$, we~get,
\begin{align*}
 (\nabla_X Q)Y &= -(\nabla_X \varphi^2)Y + [g(\nabla_X \xi,Y)\xi+ \eta(Y)\nabla_X \xi] \\
 &  = -(\nabla_X \varphi)(\varphi Y)- f(\nabla_X \varphi)Y - g(\varphi X,Y) \xi - \eta(Y)\varphi X \\
 & =\eta(Y)Q\varphi X+g(QY,\varphi X)\xi  - g(\varphi X,Y) \xi - \eta(Y)\varphi X,
 \end{align*}
which is equivalent to \eqref{E-nS-10K}.
\end{proof}

\section{Partial Ricci Flow} \label{sec:02.2}

Let $(M,g)$ be a Riemannian manifold equipped with a unit geodesic vector field $\xi$,
and ${\cal D}$ be an orthogonal distribution. 
A self-adjoint (1,1)-tensor
$R_\xi: X\to R_{X,\xi}\,\xi$ is called the Jacobi operator in the $\xi$-direction.
The \textit{partial Ricci flow} (PRF) of metrics has the view, see \cite[Section~5.3.7.2]{Rov-Wa-2021},
\begin{equation}\label{E-GF-Rmix-Phi}
 \oc g_t = -2\Ric_{\cal D}(g_t) +2\,g^\top_t ,
\end{equation}
where the (0,2)-tensor $\Ric_{\cal D}$ is dual to $R_\xi$,
$g^\top(X,Y):=g(X^\top,Y^\top)$,
$^\top$ is the $\cal D$-component of the vector,
and $\,\oc{\phantom.}\,$ denotes the $t$-derivative.
For a weak $K$-contact structure,
$(\varphi,Q,\xi,\eta,g)$, the
tensor $R_\xi$
is positive definite on $\cal D:=\ker\eta$,
\eqref{E-GF-Rmix-Phi} reduces to ODEs, see \cite{Rov-Wa-2021},
and
\begin{align}\label{E-PRF-fK}
 \oc R_\xi = 4\,R_\xi(R_\xi -\,{I}_{\ker\eta}) ,
\end{align}
where ${I}_{\ker\eta}: TM\to{\ker\eta}$ is the orthoprojector.
Thus, the $K$-contact structure ($R_\xi={I}_{\ker\eta}$) is a fixed point of~\eqref{E-GF-Rmix-Phi}.
By the theory of ODEs, there exists a unique solution
$R_{\xi,t}$ of \eqref{E-PRF-fK} for $t\in\mathbb{R}$;
thus, a solution $g_t$ of \eqref{E-GF-Rmix-Phi} exists for $t\in\mathbb{R}$ and it is unique.
The flow \eqref{E-GF-Rmix-Phi} preserves the directions of eigenvectors of $R_\xi$,
and each eigenvalue $\mu_i>0$ satisfies the ODE $\oc\mu_i=4\mu_i\,(\mu_i-1)$ with $\mu_i(0)>0$, see \cite[Section~5.3.7.2]{Rov-Wa-2021}.

\begin{theorem}[Theorem~5.72 in \cite{Rov-Wa-2021}]\label{T-PRF-g2b}
Let $M^{2n+1}(\varphi_0,Q_{\,0},\xi,\eta,g_0)$ be a weak $K$-contact manifold.
Then, there exist $\ker\eta$-adapted Riemannian metrics $g_t\ (t\in\mathbb{R})$
-- a solution of
\eqref{E-GF-Rmix-Phi} -- such that each $(\varphi_t,Q_{\,t},\xi,\eta,g_t)$ is a weak $K$-contact structure on $M$ with structural tensors defined on 
${\ker\eta}$ as
\begin{equation}\label{E-Q-phi}
 Q_{\,t}|_{\ker\eta}=R_{\xi,t}\,,
 \quad
 \varphi_t\,|_{\,{\ker\eta}}=T_{\xi}^\sharp(t) ,
\end{equation}
where the operator ${T}^{\sharp}_\xi:{\ker\eta}\to{\ker\eta}$ is defined by
$g({T}^{\sharp}_\xi X,Y)=
\tfrac12\,\eta([X,Y])$.
Moreover,~$g_t$  converges exponentially fast, as $t\to-\infty$, to a limit metric $\hat g$
with $\widehat R_\xi={I}_{\ker\eta}$ that provides an $K$-contact structure $(\hat \varphi,\xi,\eta,\hat g)$
on $M$.
\end{theorem}

By Theorem~\ref{T-PRF-g2b}, the PRF deforms a weak $K$-contact structure to a $K$-contact structure,
preserves $\xi$ and $\eta$,
the distribution ${\ker\eta}$ and its orthogonality to $\xi$.
For~example, using \eqref{E-Q-phi}, we have
\begin{align*}
 \oc g(X,\varphi Y) &=
 -2\,g(R_\xi X,\varphi Y) +2\,g(X,\varphi Y)
 =  -2\,g(\widetilde QX,\varphi Y),\\
  g(X, \oc \varphi Y) &
  = 2\,g(X, (R_\xi -I_{\ker\eta})\varphi Y)
  = 2\,g(X, \widetilde Q\,\varphi Y)
  = 2\,g(\widetilde QX,\varphi Y).
\end{align*}
Hence, $(g(X,\varphi Y))'=\oc g(X,\varphi Y)+g(X, \oc \varphi Y)=0$,
i.e., the PRF \eqref{E-GF-Rmix-Phi} preserves the equality $d\eta=F$.

\smallskip

We complement Theorem~\ref{T-PRF-g2b}  with the following result used in Theorem~\ref{P-K-S}.

\begin{Proposition}\label{L-PRF-2}
For a weak $K$-contact manifold satisfying  \eqref{Eq-2assump},
the property \eqref{E-nabla=phi-nullity}
(and hence, \eqref{E-1k-0-nullity}, see Proposition~\ref{L-cond-equiv})
remains invariant under the PRF \eqref{E-GF-Rmix-Phi}.
\end{Proposition}

\begin{proof}
In our case, $\oc g(X,Y)=-2\Ric_{\ker\eta}(X,Y) +2\,g^\top(X,Y)$.
Since $Q|_{\ker\eta}= {I}_{\ker\eta} +\widetilde Q$, where $\widetilde Q = Q-I$, by \eqref{E-Q-phi}, we have
 $\nabla_Y R_\xi =
 \nabla_Y \widetilde Q + \nabla_Y {I}_{\ker\eta}$.
Using this, $\nabla_Z\,\widetilde Q=\nabla_Z\,Q$ and \eqref{E-GF-Rmix-Phi}, we obtain
\begin{align*}
 (\nabla_Z\,\oc g)(X,Y) = -2\,g((\nabla_Z\, R_\xi)X, Y)
 +2\,g((\nabla_Z\,{I}_{\ker\eta})X, Y) = -2\,g((\nabla_Z\,Q)X,Y).
\end{align*}
We also have, see \cite[Section~5.3.7.2]{Rov-Wa-2021},
\begin{align}\label{E-PRF-fT}
 \oc{T}\,^{\sharp}_\xi =
 2\,(R_\xi - {I}_{\ker\eta}) {T}^{\sharp}_\xi .
\end{align}
In view of \eqref{E-PRF-fT} and \eqref{E-Q-phi}, we have
 $\oc\xi = \oc\eta = 0$
 and
 $\oc \varphi = \oc T\,\!_{\xi}^\sharp
 = 2\,(R_\xi - {I}_{\ker\eta}) {T}^{\sharp}_{\xi}$.
Using \eqref{E-PRF-fK}, \eqref{E-Q-phi} and
$\oc R_\xi = \oc Q + \oc{I}_{\ker\eta}=\oc Q$,
we also have
$ \oc{\varphi^2} = -\oc Q = -4\,R_\xi(R_\xi -\,{I}_{\ker\eta})$.

Let's calculate the evolution under PRF of the equality \eqref{E-nabla=phi-nullity}. We have
\begin{align} \label{Eq-ev}
 (\oc\nabla_X\,\varphi)Y
 +(\nabla_X\,\oc \varphi)Y
 =\eta(Y)\oc{\varphi^2}X
 -g(\oc{\varphi^2}\,Y,X)\xi
 -\oc g(\varphi^2Y,X)\xi .
\end{align}
Recall the general formula, e.g., \cite[Eq.~(4.8)]{Rov-Wa-2021},
\begin{equation}\label{eq2G}
 2\,g(\oc\nabla_X Y, Z) = (\nabla_X\,\oc g)(Y,Z)
 +(\nabla_Y\,\oc g)(X,Z)-(\nabla_Z\,\oc g)(X,Y).
\end{equation}
By the above, symmetry of the covariant derivative of $Q$, \eqref{E-nS-10K}
(see Proposition \ref{L-cond-equiv})
and \eqref{eq2G},
we have
\begin{align*}
\notag
 g((\oc\nabla_X\,\varphi)Y,Z)  &=
 g(\oc\nabla_X(\varphi Y),Z) + g(\oc\nabla_X Y,\varphi Z)\\
 & = \tfrac12\{(\nabla_X\,\oc g)(\varphi Y,Z)
 + (\nabla_{\varphi Y}\,\oc g)(X,Z) -(\nabla_Z\,\oc g)(X,\varphi Y)
 \notag\\
 & \quad +(\nabla_X\,\oc g)(Y,\varphi Z)
 +(\nabla_Y\,\oc g)(X,\varphi Z)
 -(\nabla_{\varphi Z}\,\oc g)(X,Y)\} \notag\\
 & = -\{g((\nabla_X Q)\varphi Y,Z)
 +g((\nabla_{\varphi Y} Q)X,Z)
 -g((\nabla_Z Q)X,\varphi Y) \notag \\
 &\quad
 +g((\nabla_X Q)Y, \varphi Z) + g((\nabla_Y Q)X,\varphi Z)
 -g((\nabla_{\varphi Z} Q)X,Y)\} \notag\\
 & =-2\,\eta(Y) g(\widetilde Q \varphi X, \varphi Z), \\
 g( (\nabla_X\,\oc\varphi )Y, Z)& = g(\nabla_X (\oc \varphi Y),Z) -g(\oc \varphi \nabla_X Y,Z) \notag \\
 & = 2\,g((\nabla_X Q)\varphi Y,Z)
 + 2\,g((R_\xi - {I}_{\ker\eta})
 (\nabla_X \varphi)Y,Z) \notag \\
&  = 2\,\eta(Z) g(\widetilde Q\varphi X, \varphi Y)
+ 2\, \eta(Y) g((R_\xi
- {I}_{\ker\eta})\varphi ^2X,Z) \notag \\
& = 2\,\eta(Z) g(\widetilde Q\varphi X, \varphi Y)
- 2\,\eta(Y) g(\widetilde Q\varphi  X,\varphi Z).
\end{align*}
Therefore, the inner product of the LHS of \eqref{Eq-ev} with $Z$
is as follows:
\begin{align}\label{E-PRF-cond1a}
 g( (\oc\nabla_X\,\varphi)Y
 +(\nabla_X\,\oc\varphi)Y, Z) =
 -4\,\eta(Y) g(\widetilde Q X, QZ)
 +2\,\eta(Z) g(\widetilde QX, QY).
\end{align}
Similarly, using $I_{\ker\eta}\,\widetilde Q=\widetilde Q$, \eqref{E-nS-10K} and \eqref{E-Q-phi}, for all $X,Y,Z\in\mathfrak{X}_M$ we calculate
\begin{align*}
 \eta(Y)g(\oc{\varphi ^2}X,Z) &=-4\, \eta(Y) \, g(R_\xi(R_\xi -\,{I}_{\ker\eta})X,Z)
 =-4\,\eta(Y) g(\widetilde Q X, QZ), \\
 -\eta(Z)g(\oc{\varphi ^2}Y,X)
 &=4\,\eta(Z)g(R_\xi(R_\xi-{I}_{\ker\eta})Y,X)
 = 4\,\eta(Z)g(Q^2Y - QY,X)
 = 4\,\eta(Z)g(\widetilde QY,QX), \\
 -\eta(Z)\oc g(\varphi ^2Y,X) &=
 2\,\eta(Z)\{g(R_\xi\,\varphi ^2 Y,X)
 -g^\top(\varphi^2Y,X)\}
 = 2\,\eta(Z)\{g(Q\varphi^2 Y,X)
 - g(\varphi^2Y, X)\}  \\
 & = - 2\,\eta(Z) g(\widetilde Q\varphi X, \varphi Y) = - 2\,\eta(Z) g(\widetilde QX, QY) .
\end{align*}
Therefore, the inner product of the RHS of \eqref{Eq-ev} with $Z$
is as follows:
\begin{align}\label{E-PRF-cond1b}
g(\eta(Y)\oc{\varphi ^2}X -g(\oc{\varphi ^2}\,Y,X)\xi
 -\oc g(\varphi ^2Y,X)\xi,\ Z) =
 -4\,\eta(Y) g(\widetilde Q X, QZ)
 +2\,\eta(Z) g(\widetilde QX, QY).
\end{align}
Since the RHS of \eqref{E-PRF-cond1a} and \eqref{E-PRF-cond1b}
coincide, \eqref{Eq-ev} holds.
Thus, \eqref{E-nabla=phi-nullity} is preserved under the PRF~\eqref{E-GF-Rmix-Phi}.
\end{proof}

\section{Proof of Main Results and Illustrative Examples}
\label{sec4}

In this section, we prove Theorems~\ref{P-K-S}--\ref{T-04b} and their corollaries and give illustrative examples.


\begin{proof}[Proof of Proposition~\ref{P-K-S}]
By Proposition \ref{L-cond-equiv}, \eqref{E-1k-0-nullity} and \eqref{E-nabla=phi-nullity} are equivalent; moreover, $M$ carries a weak $K$-contact structure.
Using the Levi-Civita connection $\nabla$ of $g$, one can rewrite the tensor $N_\varphi$ as
\begin{align}\label{4.NN}
 N_\varphi(X,Y) = ({\varphi}\nabla_Y{\varphi}
 -\nabla_{{\varphi} Y}{\varphi}) X
 -({\varphi}\nabla_X{\varphi}
 -\nabla_{{\varphi} X}{\varphi}) Y .
\end{align}
From
\eqref{4.NN} and \eqref{E-nabla=phi-nullity},
we~have
\begin{align*}
 g(\mathcal{N}^{(1)}(Y,Z),\varphi X)
 = g(\varphi (\nabla_Z \varphi) Y
 -(\nabla_{\varphi Z}\varphi ) Y
 -\varphi (\nabla_Y \varphi)Z
 +(\nabla_{\varphi Y}\varphi) Z,\ \varphi X)
 =0.
\end{align*}
Applying this in conjunction with
\eqref{E-nabla=phi-nullity} (see Proposition~\ref{L-cond-equiv}), we reduce \eqref{Eq-nabla-phi-N5} to the equality
${\cal N}^{\,(5)}=0$.

By Theorem~\ref{T-PRF-g2b} and Proposition~\ref{L-PRF-2}, there exists an $\ker\eta$-adapted family $(\varphi_t,Q_{\,t},\xi,\eta,g_t)$ of weak K-contact structures on $M$,
which converges exponentially fast, as $t\to-\infty$, to a limit K-contact structure
$(\hat\varphi,\xi,\eta,\hat g)$ satisfying~\eqref{E-1k-0-nullity}
and \eqref{E-nabla=phi-nullity}.
By well-known results, e.g., \cite[Proposition~7.6]{blair2010riemannian}, the structure
$(\hat\varphi,\xi,\eta,\hat g)$ is Sasakian.
\end{proof}

\begin{example}\rm
To illustrate Theorem~\ref{P-K-S},
we consider a
w.c.m.~manifold
$M^{2n+1}(\varphi_0,Q_0,\xi,\eta,g_0)$
with $Q_0 =\lambda\,I$ on $\ker\eta$ for some real positive $\lambda\ne1$,
satisfying the condition  \eqref{E-1k-0-nullity}.
By Theorem~\ref{T-PRF-g2b} in Section~\ref{sec:02.2},
and the theorem of uniqueness of a solution to ODE \eqref{E-PRF-fK}, we have a $\ker\eta$-conformal solution $R_{\xi,t}
=\mu(t)\,I_{\ker\eta}$.
Set $g^\top(X,Y)=g(X^\top,Y^\top)$.
Since $\oc I_{\ker\eta}=0$, we have the ODE $\oc\mu=4\mu\,(\mu-1)$ with $\mu(0)=\lambda>0$.
Thus, $\mu(t)=\frac{1}{1-C\,e^{4t}}\ (t\in\mathbb{R})$, where $C=1-\lambda^{-1}\ne0$.
Since our PRF \eqref{E-GF-Rmix-Phi} reduces to ODE, applying the theorem of uniqueness of a solution to ODE, we~have
$g\,\!^\top_t=\rho(t)\,g\,\!^\top_0$, where
$\rho(0)=1$ and
\begin{align*}
 &\oc g\,\!^\top_t =-2\mu(t)g^\top_t
 +2\,g^\top_t = 2\,(1-\mu(t))\,g^\top_t
 \implies \frac{d}{dt}\log \rho(t)=2\,(1-\mu(t))
 =  -\frac{2\,Ce^{4t}}{1-C e^{4t}} .
 \end{align*}
Thus, $g_t$ is conformal on $\ker\eta$ with the conformal factor equal to
$\rho(t)= \big(\frac{1-C e^{4t}} {1-C}\big)^{1/2}$.
By~the above,
\[
 \varphi_t=\varphi_0/\rho(t),\quad Q_{\,t}|_{\ker\eta}=\frac{1}{1-Ce^{4t}}\,I_{\ker\eta},\quad
 g_t=
\eta\otimes\eta +\rho(t)\,g^\top_0.
\]
By Theorem~\ref{P-K-S},
$(\varphi_t,Q_{\,t},\xi,\eta,g)$ $(t\in\mathbb{R})$,
is a~smooth family of weak $K$-contact structures on $M$ satisfying
\eqref{E-1k-0-nullity}, that converges exponentially fast, as $t\to-\infty$, to a Sasakian structure
$(\lambda^{-1/2}\varphi_0, \xi,\eta,\,
\eta\otimes\eta+\lambda^{1/2}g^\top_0)$.
\end{example}

\begin{proof}[Proof of Theorem~\ref{T-04}]
By Proposition~\ref{prop-5.1}(i), the eigenvalues of $\widetilde{h}$ are $\{0, \pm\sqrt{1-\kappa}\}$ and $\kappa\le1$.


$(i)$ Let $\kappa=1$, then $\widetilde{h}=h=0$.
Therefore, \eqref{E-k-mu-nullity} reduces to~\eqref{E-1k-0-nullity}.
By Proposition~\ref{C-2cond}, \eqref{E-xi-n} reduces to $\nabla\,{\xi}=-\varphi$,
and $(\varphi,Q,\xi,\eta,g)$ is a weak $K$-contact structure.
By Theorem~\ref{P-K-S}, there exists an $\ker\eta$-adapted fami\-ly $(\varphi_t,Q_{\,t},\xi,\eta,g_t)$ of weak K-contact structures on $M$,
which converges exponentially fast, as $t\to-\infty$, to a limit Sasakian structure
$(\hat\varphi,\xi,\eta,\hat g)$.

\smallskip

$(ii)$ Let $\kappa<1$.
By Proposition~\ref{prop-5.1}(ii), the contact distribution $\ker\eta$ decomposes into a direct sum
of mutually orthogonal eigen-distributions
${\cal D}_+$ and ${\cal D}_-$,
corresponding to the eigenvalues $\sqrt{1-\kappa}$ and $-\sqrt{1-\kappa}$ of $\tilde h$, respectively.
Let $X,Y\in{\cal D}_+$, or, ${\cal D}_-$, then using \eqref{E-xi-n} we get
 $g(\nabla_X\,\xi,Y)= -(1\pm\sqrt{1-\kappa})\,g(\varphi X,Y) = 0$.
Thus, $0=2\,g(QX,\varphi Y)= 2\,d\eta(X,Y)=- \eta([X,Y])$.
By~\eqref{E-k-mu-nullity}, $R_{X,Y}\,\xi=0$ for $X,Y \in\ker\eta$.
Using this for $X,Y,Z\in{\cal D}_+$, or, $X,Y,Z\in{\cal D}_-$, respectively, we have
\begin{align*}
 \notag
 0&
 = g(R_{X,Y}\,\xi,Z)
 = -(1\pm\sqrt{1-\kappa})\,(\nabla_Z{F})(X,Y)
 \mp 2\sqrt{1-\kappa}\,g(\varphi [X,Y],Z).
\end{align*}
Applying \eqref{Eq-phi02} and \eqref{Eq-phi001} (where \eqref{E-nS-10} is used) in the above relation yields
$g(\varphi [X,Y],Z)=0$. Using this and $\eta([X,Y])=0$, we get $[X,Y]\in {\cal D}_+$, or, $[X,Y]\in {\cal D}_-$\,, respectively, which shows that the distributions ${\cal D}_+$ and ${\cal D}_-$ are involutive.

Fix $X,Z\in{\cal D}_-$ and $Y\in {\cal D}_+$, then using the fact $R_{e_1,e_2}\,\xi=0\ (e_1, e_2 \in\ker\eta)$, we have
\begin{align*}
 0&= g(R_{X,Y}\,\xi ,Z) =
 (1+\sqrt{1-\kappa})\big\{ (\nabla_X {F})(Y,Z)
 + (\nabla_Y {F})(Z,X)\big\}
 - 2\,\sqrt{1-\kappa}\, g(\varphi \nabla_Y X,Z), \\
 & = (1+\sqrt{1-\kappa}) (\nabla_Z {F})(X,Y)- 2\,\sqrt{1-\kappa}\, g(\varphi \nabla_Y X,Z).
\end{align*}
Using \eqref{Eq-phi02}, gives
$(\nabla_Z {F})(X,Y)=0$, which implies $g(\nabla_Y X,\varphi Z)=0$ for $X,Z\in{\cal D}_+$ and $Y\in{\cal D}_-$.
Hence, $\nabla_Y X\
(X\in{\cal D}_+,\,Y\in{\cal D}_-)$
is orthogonal to ${\cal D}_-$.
Ta\-king  \(Y,Z\in{\cal D}_+\) and \(X\in{\cal D}_-\), and applying \eqref{E-xi-n}, gives
$g(\nabla_Y Z, \xi) = -g(\nabla_Y \xi, Z) = (1+ \sqrt{1-\kappa})g(\varphi Y, Z) = 0$
and
$g(\nabla_Y Z, X) = -g(Z, \nabla_Y X) = 0$.
By~ the above, \(\nabla_Y Z \perp {\cal D}_- \oplus  \{\xi\}\), hence
\({\cal D}_+\) is totally geodesic. Applying the same procedure for $X,Z\in{\cal D}_+$ and $Y\in {\cal D}_-$ and using \eqref{Eq-phi02}, we conclude that ${\cal D}_-$ is totally geodesic.
\end{proof}

\begin{proof}[Proof of Theorem~\ref{thm6}]
$(i)$ This follows from the conditions and Definition~\ref{D-02}.

$(ii)$  Note that $g'(X,\xi)=\eta(\xi)=g(X,\xi)$ and
$g'(X,Y^\top)={\lambda}\,g(X,Y^\top)$, where $Y^\top$ is the $\ker\eta$-component of $Y$.
Let $\nabla'$ and $\nabla$ be the Levi-Civita connections of the metrics $g'$ and $g$, respectively.
In the Koszul formula for $g'$,  we expand each term using
$g'(X,Y) = {\lambda}\,g(X,Y) + (1-{\lambda})\,\eta(X)\eta(Y)$, see \eqref{E-cont-lambda}, to obtain
\begin{align*}
 & 2\,g'(\nabla'_X Y, Z)
 = 2\,{\lambda}\,g(\nabla_X Y,Z)
 +2(1{-}{\lambda})\,\eta(\nabla_X Y) \eta(Z)
 +(1{-}{\lambda})\big\{(\nabla_X \eta)(Y)\eta(Z)  \\
 &+\eta(Y)(\nabla_X \eta)(Z)
 +(\nabla_Y\eta)(X)\eta(Z) +\eta(X)(\nabla_Y \eta)(Z) -(\nabla_Z \eta)(X)\eta(Y)-\eta(X)(\nabla_Z \eta)(Y)\big\}.
\end{align*}
Simplifying the above by using
the relation $(\nabla_X\,\eta)(Y)= g(\nabla_X\,\xi,Y) =-g(\varphi X +\varphi\,\widetilde h X,Y)$ and the fact $g'(\nabla_X Y,Z) = {\lambda}\,g(\nabla_X Y,Z)
 +(1-{\lambda})\,\eta(\nabla_X Y) \eta(Z) $, we have
\begin{align*}
 & 2\,g'(\nabla'_X Y, Z) = 2\,g'(\nabla_X Y, Z) +(1-{\lambda})\big\{-\eta(Z) g(\varphi X +\varphi\,\widetilde h X,Y)
 -\eta(Y) g(\varphi X +\varphi\,\widetilde h X,Z) \\
 & -\eta(Z) g(\varphi Y {+}\varphi\,\widetilde h Y,X) {-}\eta(X)\,g(\varphi Y {+}\varphi\,\widetilde h Y, Z)
 {+}\eta(Y)\,g(\varphi Z {+}\varphi\,\widetilde h Z,X) {+}\eta(X)\,g(\varphi Z {+}\varphi\,\widetilde h Z,Y) \big\}.
\end{align*}
With further simplification, we obtain
\begin{align}\label{E-nabla'}
 \nabla'_X Y =\nabla_XY+ (1-{\lambda})\,\big\{g(\widetilde hX,\varphi Y)\xi-\eta(Y) \varphi X-\eta(X) \varphi Y\big\}.
\end{align}
 From \eqref{E-k-mu-nullity} and the definition of $Q$, we have
$ R_{\,\xi,X}Y = \kappa\{ g(QX,Y)\xi
- \eta(Y)QX \}
 + \mu \{ g(hX,Y)\xi - \eta(Y)hX \}.$
Using the previous equation and \eqref{E-05b} gives
\begin{equation*}
(\nabla_{hX}\varphi)Y
= {\lambda}^4(1-\kappa )\big[ g(X,Y)\xi-\eta(Y)X \big]
+ {\lambda}^2\big[\eta(Y)hX - g(hX,Y)\xi\big].
\end{equation*}
Next, in the above equation, we substitute $hX$ in place of $X$ and use the equality
 $h^2X =(1-\kappa)\big\{Q^2X -\eta(X) \xi\big\},$
see the proof of Proposition~\ref{prop-5.1}, to obtain
\begin{equation}\label{Eq-varphi}
 (\nabla_X \varphi)Y = {\lambda}^2[g(X,Y)\xi-\eta(Y) X] +g(hX,Y)\xi-\eta(Y) hX.
\end{equation}
We set
 $A(X,Y) := g(\widetilde hX,\varphi Y)\xi - \eta(Y)\varphi X - \eta(X)\varphi Y$,
and using this along with \eqref{E-nabla'},  find
\[
 R'_{X,Y}\,\xi = R_{X,Y}\,\xi
 +(1-{\lambda})\big\{(\nabla_X A)(Y,\xi)
 -(\nabla_Y A)(X,\xi) +A(X,A(Y,\xi)) -A(Y,A(X,\xi)) \big\}.
\]
By \eqref{Eq-varphi}, we have the following expressions:
\begin{align*}
 (\nabla_X A)(Y,\xi) =& -2g(hX,Y)\xi -(2-\kappa) {\lambda}^2 g(X,Y)\xi
 +[(1-\kappa){\lambda}^2+1]\eta(X)\eta(Y)\xi\\
 & +2{\lambda}^2\eta(Y)X +2\eta(Y) hX ,\\
 A(X, A(Y,\xi)) =& -A(X, \varphi Y) = -g( hX, Y)\,\xi+\eta(X)\varphi^2 Y.
\end{align*}
Using these two equations and the definition of $\varphi'$,
$h = \tfrac12\mathcal{L}_{\xi}\,\varphi = {\lambda}\,\tfrac12\mathcal{L}_{\xi}\,\varphi'= {\lambda}\,h'$, and \eqref{E-k-mu-nullity}, we have
\begin{align*}
 R'_{X,Y}\,\xi &= R_{X,Y}\,\xi + (1-{\lambda}) \big\{3(\eta(X) \varphi^2Y-\eta(Y) \varphi^2X ) + 2(\eta(Y) hX -\eta(X) hY)\big\} \\
 & ={\lambda}^2(\kappa+3(1-{\lambda}))\big\{ \eta(X)\varphi'^2Y -\eta(Y)\varphi'^2X \big\}
 + {\lambda}(\mu+ 2(1-{\lambda}))\,\big\{ \eta(Y) h'X -\eta(X) h'Y\big\} .
\end{align*}
Taking $\kappa'= \lambda^2(\kappa+3(1-\lambda))$ and $\mu' =\lambda(\mu+2(1-\lambda))$, we see that $(\varphi',\xi,\eta,g')$ is a $(\kappa',\mu')$-structure.
\end{proof}


\begin{proof}[Proof of Theorem~\ref{Thm-lie}]
By \eqref{E-nS-10} and Example~\ref{Ea-dim3}$(ii)$, $Q|_{\ker\eta}={\lambda}^2\,I_{\,\ker\eta}$ for some real ${\lambda}>0$.
By Theorem~\ref{thm6},
$(\varphi',\xi,\eta,g')$ is a contact metric structure with
$\varphi'={\lambda}^{-1}\varphi$ and
$g'={\lambda} g+(1-{\lambda})\eta\otimes\eta$,
satisfying the $(\kappa',\mu')$-nullity condition with
$\kappa'= \kappa+3{\lambda}^2(1-{\lambda})$ and
$\mu'=\mu+2{\lambda}(1-{\lambda})$.
Applying \cite[Theorem~3]{blair1994}
we conclude that $M^3(\varphi',\xi,\eta,g')$ is locally isometric to any of the following unimodular Lie groups $G$ with a left-invariant~metric:

$SU(2)$ or $SO(3)$, when $c_2>0$ and $c_3>0$;

$SL(2,\mathbb{R})$, or, $O(1,2)$, when $c_2<0$, and $c_3>0$, or, $c_2<0$ and $c_3<0$;

$E(2)$, when $c_2=0$ and $\mu<2$, including a flat structure when $\mu=0$;

$E(1,1)$, when $c_3=0$ and $\mu>2$;
where
$c_2 = 1-\sqrt{1-\kappa}-\tfrac12\,\mu$
and
$c_3=1+\sqrt{1-\kappa}-\tfrac12\,\mu$.

\noindent
Here, $G$ is the Lie group with a 3-dimensional Lie algebra $\mathfrak{g}$,
having a basis $\{E_1,E_2,\xi\}$, and the following Lie brackets:
 $[E_1,E_2]= 2\,{\lambda}\,\xi$,
 $[E_2,\xi]=c_2 E_1$, and
 $[\xi,E_1]=c_3 E_2$,
where ${\lambda}>0,c_2,c_3$ are real
and $\mu = 2-(c_2+c_3)$.


Define an almost w.c.m. structure $(\varphi, Q,\xi,\eta,g)$ on $G$ as~follows:
\begin{align*}
 &\varphi E_1= {\lambda} E_2, \quad \varphi E_2=-{\lambda} E_1, \quad \varphi\,\xi=0,
 \quad QE_1 ={\lambda}^2 E_1, \quad QE_2 ={\lambda}^2 E_2, \quad Q\,\xi=\xi,
\end{align*}
$g$ is the left-invariant Riemannian metric such that $\{E_1,E_2,\xi\}$ is an orthonormal frame and $\eta$ is the dual $1$-form of $\xi$ with respect to $g$.
Note that the images of $E_1,E_2$ under $\widetilde h$ are
$
  \widetilde hE_1= \tfrac1{2{\lambda}}\,(c_3-c_2)\,E_1, \, \widetilde hE_2=-\tfrac1{2{\lambda}}\,(c_3-c_2)\,E_2,
 $
indicating ${\cal D}_+$ is spanned by $E_1$ and ${\cal D}_-$ by $E_2$.
%
The following relations are satisfied:
\begin{align*}
 &\nabla_{E_1} E_2 =\tfrac{1}{2}(2{\lambda}-c_2+c_3)\,{\xi},\quad
 \nabla_{E_1}\,\xi= \tfrac{1}{2}(c_2-c_3-2{\lambda})\,E_2, \quad
 \nabla_{E_2}\,\xi =\tfrac{1}{2}(2{\lambda}+c_2-c_3)\,E_1, \\
 & \nabla_{\xi} E_1 = \tfrac{1}{2}(c_2+c_3-2{\lambda})E_2, \quad
 \nabla_{\xi} E_2 = \tfrac{1}{2}(2{\lambda}-c_2-c_3)E_1, \quad \nabla_{E_1} E_1 = \nabla_{E_2}E_2 =\nabla_{\xi}\,\xi = 0.
\end{align*}
Using the above relations, we find the following components of the curvature tensor:
\begin{align*}
& R_{\,E_1, \, E_2}\,\xi=0, \quad
 R_{\,E_1,\, \xi}\,\xi=-\big\{1-\tfrac1{4{\lambda}^2}\,(c_3-c_2)^2\big\}\,\varphi^2 E_1 + \tfrac{1}{{\lambda}}(2{\lambda}- {c_2} - {c_3})\,hE_1, \\
&R_{\,E_2,\, \xi}\,\xi= -\big\{1-\tfrac1{4{\lambda}^2}(c_3-c_2)^2\big\} \varphi^2 E_2 + \tfrac{1}{{\lambda}}(2{\lambda}-{c_2} -{c_3})\,hE_2.
\end{align*}
Choosing
 $\kappa= 1-\tfrac1{4{\lambda}^2}(c_3-c_2)^2
 \le 1$
 and
$\mu=\tfrac{1}{\lambda}(2{\lambda}-c_2-c_3)$,
we get
the $(\kappa,\mu)$-nullity condition on $G$.
Moreover, \eqref{Eq-2assump} and \eqref{E-nS-10} are valid.
Since $\kappa<1$ for $c_2\ne c_3$, we have the desired~structure.
\end{proof}

\begin{example}\rm
We illustrate Theorem~\ref{T-04}$(ii)$  by representing a w.c.m. manifold satisfying the conditions \eqref{E-k-mu-nullity}
with $\kappa<1$, \eqref{Eq-2assump} and \eqref{E-nS-10}.
For $n=1$, this also illustrates Corollary~\ref{Thm-lie} since the manifold is locally isomorphic to $E(1,1)$.
Let $\mathfrak{g}$ be a $(2n+1)$-dimensional Lie algebra with a basis $\{E_1,\ldots,E_n,{\cal E}_1,\ldots,{\cal E}_n,\xi\}$
and the Lie bracket (satisfying the Jacobi identity) defined as follows for a real ${\lambda}>0$ (for ${\lambda}=1$, see
\cite{CMdiT-2007}):
\begin{align}\label{E-Lie_ex4}
\notag
 &[E_a,E_b]=0 \quad \text{for any } a,b\in\{1,\ldots,n\},
 \qquad [{\cal E}_a,{\cal E}_b]=0 \quad \text{for any } a\ne 2,\\
\notag
 &[{\cal E}_2,{\cal E}_b]=2\,{\lambda}\,{\cal E}_b\  \text{ for any } b\ne 2, \quad [E_1,{\cal E}_1]=2\,{\lambda}\,({\xi}-E_2), \quad [E_1,{\cal E}_b]=0\  \text{ for any } b\ge 2,\\
\notag
 &[E_a,{\cal E}_b]=2\,{\lambda}\,\delta_{ab}({\xi}-E_2) \quad \text{for any } a,b\ge 3,
\qquad [E_2,{\cal E}_b]=2\,{\lambda}\,E_b \text{ for any } b\ne 2,\\
\notag
 &[E_2,{\cal E}_2]=2\,{\lambda}\,{\xi}, \qquad
 [E_a,{\cal E}_1]=[E_a,{\cal E}_2]=0\ \text{ for } a\ge 3, \\
 &[\xi,E_b]=0, \qquad
 [\xi,{\cal E}_b]=2\,{\lambda}\,E_b
\quad \text{for any } b\in\{1,\ldots,n\}.
\end{align}
Define a weak almost contact structure
\((\varphi,Q,\xi,\eta)\)
on a Lie group $G$ of the Lie algebra $\mathfrak g$ by
\begin{align}\label{E-f_ex4}
 \varphi\,E_a={\lambda} {\cal E}_a, \quad
 \varphi\,{\cal E}_a=-{\lambda} E_a
 \quad (1\le a\le n),
 \quad \varphi\,\xi=0,
 \quad
 Q=\lambda^2 I+(1-\lambda^2)\,\eta\otimes\xi,
 \quad
 \eta=\xi^\flat.
\end{align}
Let $g$ be a left-invariant Riemannian metric on $G$ such that $\{E_1,\ldots,E_n,\,{\cal E}_1,\ldots,{\cal E}_n,\,\xi\}$ is an orthonormal frame.
Then $(\varphi,Q,\xi,\eta,g)$ satisfies \eqref{Eq-2assump} and \eqref{E-nS-10}.
Note that $n$-dimensional distributions ${\cal D}_+=Span\{E_1,\ldots,E_n\}$ and
 ${\cal D}_-=Span\{{\cal E}_1,\ldots,{\cal E}_n\}$ are
 involutive.

Using \eqref{E-Lie_ex4} and \eqref{E-f_ex4}, we have
the following non-zero derivatives of the frame:
\begin{align}\label{E-nabla_ex4}
\notag
&
\nabla_{E_1}\xi=-2\,{\lambda}\, {\cal E}_1, \quad 
\nabla_{E_2}\xi=-2\,{\lambda}\,{\cal E}_2, \quad
\nabla_{E_a}{\cal E}_b=2\,{\lambda}\delta_{ab}\,{\xi}, \\
&\nabla_{E_a}\xi=-2{\lambda}\,{\cal E}_a, \quad
\nabla_{{\cal E}_1}{\cal E}_1=2\,{\lambda}\,{\cal E}_2,\quad
\nabla_{{\cal E}_a}{\cal E}_2=-2{\lambda}\,{\cal E}_a,\quad
\nabla_{{\cal E}_a}{\cal E}_b=2{\lambda}\delta_{ab}\,{\cal E}_2,
\end{align}
for all $a,b\in\{3,\ldots,n\}$, from which we find the non-zero components of the curvature tensor:
\begin{align*}
& R_{E_1,\xi}\,\xi = 4{\lambda}^2 E_1, \ \quad
  R_{E_2,\xi}\,\xi = 4{\lambda}^2 E_2, \quad
  R_{E_a,\xi}\,\xi = 4{\lambda}^2 E_a, \\
& R_{{\cal E}_1,\xi}\,\xi = -4{\lambda}^2 {\cal E}_1, \quad
  R_{{\cal E}_2,\xi}\,\xi = -4{\lambda}^2 {\cal E}_2, \quad
  R_{{\cal E}_a,\xi}\,\xi = -4{\lambda}^2 {\cal E}_a.
\end{align*}
We also have $\widetilde h(E_a)=E_a$ and
$\widetilde h({\cal E}_a)=-{\cal E}_a$ for $a\in\{1,\ldots,n\}$.
Let us show that $d\eta(X,Y) =g(X,\varphi Y)$
is true for the orthonormal basis
$\{E_a, {\cal{E}}_a, \xi\}_{1 \le a \le n}$.
Obviously, $d \eta(X,Y)=0=g(X, \varphi Y)$ holds for $X=\xi$ or
$Y=\xi$. Using \eqref{E-Lie_ex4}, for $X,Y\in\ker\eta$, we have
$d\eta(E_a, E_b) =d\eta({\cal E}_a, {\cal E}_b)=0\ (1\le a,b\le n)$
and $d\eta(E_a, {\cal{E}}_b) = \lambda\,\delta_{ab}$.

Also, using the orthonormality of $\{E_1,\ldots,E_n,{\cal E}_1,\ldots,{\cal E}_n,\xi\}$ and the definition of $\varphi$, we have
\begin{align*}
    g(E_a, \varphi E_b)= g({\cal E}_a, \varphi  {\cal E}_b) =0, \quad g(E_a, \varphi {\cal E}_b) = -\lambda \, \delta_{ab}.
\end{align*}
Therefore, $(\varphi,Q,\xi,\eta,g)$ is a w.c.m. structure on $G$
satisfying the condition
\eqref{E-k-mu-nullity} with $\kappa=0$ and $\mu=4$.
 Moreover, ${\cal D}_+$ and  ${\cal D}_-$ are conjugate Legendrian foliations, which together with $\xi$ form a web of totally geodesic foliations.

This manifold does not split along the
distributions
${\cal D}_+$, ${\cal D}_- $, and $\{\xi\}$.
 Indeed,
${\cal D}_- \oplus \{\xi\}$
is not closed under the Lie bracket~\eqref{E-Lie_ex4}; moreover, by \eqref{E-nabla_ex4},
${\cal D}_+ \oplus \{\xi\}$ fails to be totally geodesic.
\end{example}

\begin{proof}[Proof of Theorem~\ref{T-04b}]
It was shown in Theorem~\ref{T-04}$(ii)$ that ${\cal D}_\pm$ are totally geodesic distributions.
For $X\in{\mathcal D}^-$, $\nabla_{[\xi,X]}\,\xi=-R_{\,\xi,X}\,\xi=0$; thus, $\widetilde h\,[\xi,X]=-[\xi,X]$, i.e., $[\xi,X]\in{\mathcal D}^-$. Combining this with the involutivity of ${\mathcal D}^-$, we find that ${\mathcal D}^-\oplus\{\xi\}$ is involutive.
Consider a foliated chart $U$ with coordinates $x_1,\ldots,x_{2n+1}$ such that $\{\partial_j\}_{j>n}$ is a local basis of ${\mathcal D}^-\oplus\{\xi\}$.
There exist smooth functions $c^j_a$ on $U$, $a\in\{1,\dots ,n\}$ and $ j\in \{n+1,\dots,2n+1\}$ such that
$X_a=\partial_a+\sum_{j>n}c^j_a \partial_j$
is a basis of ${\mathcal D}^+$.
Since $[\partial_j, X_a]\in{\mathcal D}^-\oplus\{\xi\}$ for $j>n$ and $a\le n$, we can write
$[\partial_j, X_a]=X+\sigma\,\xi$, where $X\in{\mathcal D}^-$ and $\sigma\in C^\infty(U)$.
By~$\nabla_{[\partial_j, X_a]}\,\xi = \nabla_X\,\xi   + \sigma\nabla_{\xi}\,\xi = 0$,
$\xi$ is parallel along
$[\partial_j, X_a]$.
Then, using \eqref{E-RXY-xi} and~\eqref{E-xi-n},
\begin{align*}
 0 =R_{\,\partial_j, X_b}\, \xi =
  \nabla_{\partial_j}\nabla_{X_b}\,\xi
 -\nabla_{X_b}\nabla_{\partial_j}\,\xi
 -\nabla_{[\partial_j, X_b]}\,\xi
 = -2\,\nabla_{\partial_j}(\varphi X_b),
\end{align*}
and $\nabla_{\varphi X_a}(\varphi X_b)=0$,
since $\varphi X_a\in{\mathcal D}^-$.
Thus, ${\mathcal D}^-\oplus\{\xi\}$ determines a flat, totally geodesic foliation.
Let's show that the distributions ${\cal D}_+$ and ${\cal D}_-\oplus [\xi]$ are parallel.
By \eqref{E-xi-n} for $X,Z\in {\mathcal D}^-$ and  $Y\in {\mathcal D}^+$, we find
\begin{align*}
 g(R_{X,Y}\,\xi, Z) &= g(\varphi \widetilde{h}\nabla_X Y,Z)- g(\varphi \widetilde{h}\nabla_Y X,Z)
 -2g((\nabla_X \varphi )Y,Z) -g(\varphi \nabla_X Y,Z) -g(\varphi \nabla_Y X,Z)  \notag \\
 & = 2 (\nabla_X {F})(Y,Z)-2g(\varphi \nabla_Y X,Z).
\end{align*}
This yields the equality
$(\nabla_X{F})(Y,Z)=g(\varphi\nabla_Y X,Z)$.
From \eqref{Eq-phi03} (where \eqref{E-nS-10} is used), we find \(g(\nabla_Y X, \varphi Z) = 0\) for $X,Z\in {\mathcal D}^-$ and  $Y\in {\mathcal D}^+$. Since \(\varphi\) is an isomorphism from \({\cal D}_-\) onto \({\cal D}_+\), this implies that \(\nabla_Y X\) is orthogonal to \({\cal D}_+\), thus,
\(\nabla_Y X\) lies in ${\cal D}_- \oplus \{\xi\}\), indicating the distribution is parallel.
Using the total geodesy of ${\cal D}_- \oplus \{\xi\}$, we have $g(\nabla_YX, Z) =-g(X, \nabla_Y Z)=0$ for $X\in{\mathcal D}^+$ and $Y,Z\in{\cal D}_-$, indicating ${\cal D}_+$ is a parallel distribution.
By the de Rham Decomposition Theorem and the fact that \({\cal D}_- \oplus \{\xi\}\) determines a flat totally geodesic foliation, \(M\) splits with one of the factors locally isometric to $\mathbb{R}^{n+1}$.
By Theorem~\ref{T-04}(ii), $({\cal D}_+,{\cal D}_-)$ is a bi-Legendrian structure.

\smallskip

In the rest of the proof, all calculations are done for $X,Y,Z,W\in{\cal D}_+$.
Using \eqref{E-xi-n}, we have
\[
 g((\nabla_X\,\varphi)Y,\xi)= -g((\nabla_X\,\varphi )\xi,Y)
 = g(\varphi\nabla_X\,\xi, Y) = 2\,g(\varphi X,\varphi Y)
 = 2\,g(QX,Y),
\]
hence, from \eqref{Eq-phi02}, we acquire
\begin{align}\label{E-phi2}
 (\nabla_X\,\varphi)Y=2\,g(QX,Y)\xi .
\end{align}
Since ${\cal D}_+$ is totally geodesic, we have $\nabla_W Z\in{\cal D}_+$.
Using this and \eqref{E-phi2}, we~get
\begin{align}\label{E-RXY1}
\notag
 & g(\nabla_X \nabla_Y \varphi Z,\varphi W)-g(\nabla_X \nabla_Y Z, W)
 = g(\nabla_X (2g(QY,Z)\xi +\varphi\nabla_Y Z),\varphi W)  \\
 &-g(\nabla_X \nabla_Y Z, W) = -4\,g(QY,Z)g(QX,W) +g(\nabla_X\nabla_Y Z, \widetilde QW) .
\end{align}
From \eqref{E-phi2} we have the following:
\begin{align}\label{E-RXY2}
\notag
 g(\nabla_{[X,Y]}\,\varphi Z, \varphi W)
 {-}g(\nabla_{[X,Y]} Z,W)
 &=2\,g(Q[X,Y],Z)\,g(\xi,\varphi W) {+} g(\varphi\nabla_{[X,Y]}Z,\varphi W) {-}g(\nabla_{[X,Y]} Z,W) \\
 &=g(\nabla_{[X,Y]}Z, Q W) -g(\nabla_{[X,Y]} Z,W)
  = g(\nabla_{[X,Y]} Z,\widetilde QW).
\end{align}
Comparing \eqref{E-RXY1} with \eqref{E-RXY2}
and using $Q=I+\widetilde Q$, we get
\begin{align*}
 g(R_{X,Y}\varphi Z,\varphi W) - g(R_{X,Y}\,Z, QW)
 = 4\, g(QX,Z)\,g(Y,QW) - 4\, g(X,QW)\,g(QY,Z) .
\end{align*}
As the manifold splits as a product of manifolds, 
$g(R_{X,Y} \varphi Z, \varphi W)=0$ for $X,Y,Z,W \in {\cal D}_+$, indicating that after replacing $QW$ by $W$, we have
\begin{align}\label{E-RXY3}
    g(R_{X,Y}\,Z,W) = 4\,\{g(QY,Z)\,g(X,W) -g(QX,Z)\,g(Y,W)\} \qquad
 (X,Y,Z,W\in{\cal D}_+).
\end{align}
Next, using the skew-symmetry of the Riemannian curvature tensor, we have
\[
 0=g(R_{X,Y}\,Z,Z)= 4\,\{g(QY,Z)\,g(X,Z) -g(QX,Z)\,g(Y,Z)\}.
\]
Taking $Z\in\mathcal{D}_+$ as a unit vector and setting $X=Z$ yields
\begin{align*}
    g(QY,Z)=g(QZ,Z)\,g(Y,Z)\qquad (Y\in\mathcal D_+).
\end{align*}
By symmetry of $Q$, we have
$QZ = g(QZ,Z)\,Z$.
Thus, every nonzero vector in $\mathcal{D}_+$ is an eigenvector of $Q$.
Hence, $Q|_{\mathcal{D}_+} =\lambda\, I|_{\mathcal{D}_+}$ for some smooth function $\lambda$ on $M$. Since $\varphi$ and $Q$ commute and $\varphi (\mathcal{D}_-)=\mathcal{D}_+$, for $V \in \mathcal{D}_ -$ we have
\[
\varphi QX =Q\varphi V = \lambda\, \varphi V \implies \varphi (QV -\lambda \, V)=0.
\]
Since, $\varphi$ is injective on $\ker\eta$, it follows that $QV =\lambda \, V$. From this and above, $Q|_{\ker\eta} =\lambda\, I|_{\ker\eta}$.
The~conditions \eqref{Eq-2assump} and \eqref{E-nS-10} force $\lambda $ to be a positive real constant. Consequently, \eqref{E-RXY3} simplifies~as
\[
g(R_{X,Y}\,Z,W) = 4\lambda\,\{g(Y,Z)\,g(X,W) -g(X,Z)\,g(Y,W)\} \qquad
 (X,Y,Z,W\in{\cal D}_+).
\]
Hence, the integral manifolds of $\mathcal{D}_+$ have constant sectional curvature $1/\sqrt{4\lambda}$ and the manifold is locally isometric to $\mathbb{R}^{n+1}\times \mathbb{S}^n(4\lambda)$.
\end{proof}

\begin{proof}[Proof of Corollary~\ref{cor4.1}]
Suppose, on the contrary, that a w.c.m. manifold of dimension $\ge 5$ is flat.
Since the manifold satisfies \eqref{E-RXY-xi}
along with \eqref{Eq-2assump} and \eqref{E-nS-10},
the equation~\eqref{E-RXY3} is valid
and reduces to
\begin{align}\label{E-cor1}
 g(Y,Z)\,g(X,W) = g(Y,W)\,g(X,Z) \quad (X,Y,Z,W\in {\cal D}_+).
\end{align}
Choose $Z=Y$ and $W=X$, then 
\eqref{E-cor1} yields
$g(X,Y)^2= g(X,X)\,g(Y,Y)$.
By the Cauchy-Schwarz inequality, this equality holds if and only if $X$ and $Y$ are linearly dependent. This contradicts our assumption
$\dim{\cal D}_+=n\ge 2$, that is, $X$ and $Y$ can be chosen to be linearly independent.
\end{proof}



\begin{example} \rm \label{Ea-dim3}
$(i)$
Let $M^{2n+1}(\varphi,Q,\xi,\eta,g)$  be a w.c.m. manifold with the curvature-related property \eqref{E-RXY-xi}, and
the tensor $Q|_{\,\ker\eta}$ be a scalar multiple of the identity map, i.e., $Q|_{\ker\eta}=\lambda\,I_{\,\ker\eta}$ for some $\lambda\in\mathbb{R}_{+}$. Then, the conditions \eqref{Eq-2assump} and \eqref{E-nS-10} are true, and Theorem~\ref{T-04b} can be applied. By \eqref{E-RXY3},
the integral manifolds of ${\cal D}_+$ have constant sectional curvature $4\lambda>0$,
and the manifold is locally isometric to
$\mathbb{R}^{n+1}\times\mathbb{S}^{n}(4\lambda)$.

$(ii)$ For $\dim M=3$, we illustrate Theorem~\ref{T-04b} with an example with fewer assumptions, which agrees with the result
of \cite[Section~7.5]{blair2010riemannian}.
Let a w.c.m. manifold \( M^3(\varphi, Q, \xi, \eta, g) \) satisfy conditions \eqref{E-RXY-xi} and \eqref{E-nS-10}.
Then $Q|_{\ker\eta}={\lambda}^2 I|_{\ker\eta}$ for some real $\lambda>0$,
see Example~\ref{Ea-dim3}$(ii)$.
Hence, the condition \eqref{Eq-2assump} is true. By Theorem~\ref{T-04b}, our manifold is locally isometric to $\mathbb{R}^2\times\mathbb{R}$; hence, the metric $g$ is flat. We~have the following Lie bracket relations:
\[
 [e_2,\xi]= 0,\quad
 [\xi, e_1]=2\,{\lambda}\,e_2,\quad
 [e_1,e_2]=2\,{\lambda}\,\xi,
\]
which determine locally the Lie group $E(2)$ with a left-invariant metric.
The~following equations are true:
 $\nabla_{e_1}\xi =-2\,{\lambda}\,e_2, \ \nabla_{e_2}\,\xi = 0$, and $\nabla_{\xi}\,\xi=0$.
Thus, $\xi$ rotates along the ${\cal D}_-$-curves with the speed equal to $2\,{\lambda}$.
\end{example}

\section{Auxiliary Results}
\label{sec:3}

Here, we prove a few propositions on the curvature and structure tensors of w.c.m. manifolds satisfying \eqref{Eq-2assump} and \eqref{E-nS-10}.
First, we generalize \cite[Proposition~7.1]{blair2010riemannian}.

\begin{Proposition}
For a w.c.m. manifold satisfying
\eqref{Eq-2assump}, the following is true:
\begin{align}\label{E-37}
 & QR_{\,\xi, X}\,\xi-\varphi R_{\,\xi,\,\varphi X}\,\xi
 = 2\,(h^2+Q\varphi^2) X.
\end{align}
\end{Proposition}

\begin{proof}
Let's compute $R_{\xi\, X}\,\xi$ using \eqref{E-xi-n}
(where \eqref{Eq-2assump} is used)
together with $\nabla_{\xi}\,\xi=0$:
\begin{align}\label{E-38}
 R_{\xi, X}\,\xi
 & = -\nabla_\xi ( \varphi X +\varphi\,Q^{-1}hX)
 +\varphi[\xi,X] + \varphi\,Q^{-1}h\,[\xi,X].
\end{align}
Applying $\varphi$ to both sides of \eqref{E-38} and then recalling Proposition~\ref{C-2cond} and \eqref{E-xi-n}, we get
\begin{align*}
 \varphi R_{\,\xi, X}\,\xi
 &=-\varphi\nabla_\xi(\varphi X +\varphi\,Q^{-1}hX) +\varphi^2[\xi,X]+\varphi^2\,Q^{-1}h[\xi,X]\notag\\
 & = (\nabla_\xi\,h)X +Q\nabla_X\,\xi +h\nabla_X \xi
 -\eta((\nabla_\xi\,h)X)\xi -\eta(\nabla_X\,\xi)\xi
 -\eta(Q^{-1}h\nabla_X\,\xi).
\end{align*}
Using \eqref{E-xi-n} and the equality $\eta((\nabla_\xi\,h)X)=0$ (obtained by taking the covariant derivative of $g(hX,\xi)=0$ along $\xi)$ in the above equation, we get the relation
\begin{align}\label{E-36}
 (\nabla_{\xi}\,h)X=\varphi R_{\,\xi, X}\,\xi
 +Q\varphi X -\varphi\,Q^{-1}h^2X .
\end{align}
From \eqref{E-36} we get the following:
\begin{align*}
 QR_{\,\xi, X}\,\xi =-\varphi (\nabla_\xi\,h)X +Q\varphi^2X +h^2X,\quad
 \varphi R_{\,\xi, \varphi X}\,\xi =
 -\varphi(\nabla_\xi\,h)X-Q\varphi^2X-h^2X.
\end{align*}
The~combination of these two relations yields  \eqref{E-37}.
\end{proof}

The following result generalizes \cite[Lemma~7.3]{blair2010riemannian} of Olszak (1979), and is used in Proposition~\ref{P-06}.

\begin{Proposition}
For a w.c.m. manifold satisfying \eqref{Eq-2assump} and \eqref{E-nS-10}, the following is true:
\begin{align}\label{E-04}
\notag
 (\nabla_X \varphi) Y  +(\nabla_{\varphi X}\varphi)\varphi Y
 =&\,2\,g( QX,QY)\xi -\eta(Y) [Q(Q+h)X+\eta(Y) \xi]  \\
 & -\tfrac{1}{2}\,g((Q+h)Y,\widetilde QX)\,\xi -\tfrac{1}{2}\big[\widetilde Q (\nabla_X\,\varphi)Y +(\nabla_{\widetilde QX}\,\varphi)Y \big] .
\end{align}
 \end{Proposition}

\begin{proof}
Using \eqref{E-nabla-phi-Phi}, \eqref{E-xi-n} and \eqref{E-nS-10b}, we have
\begin{align}
 (\nabla_X {F})(\varphi Y,Z)  - (\nabla_X {F})( Y,\varphi Z)
 & = -\, \eta(Y)g((Q+h)X,\varphi Z)
 -\eta(Z)g((Q+h)X,\varphi Y) \label{E-39}.
\end{align}
Replacing $Z$ by $\varphi Z$ in \eqref{E-39}, we acquire
\begin{align}
 (\nabla_X {F})(\varphi Y, \varphi Z) - (\nabla_X {F})(Y,\varphi^2 Z)
 = \eta(Y)[g((Q+h)X,QZ) -\eta(X) \eta(Z)].
 \label{E-40}
\end{align}
 A simple computation gives
 $(\nabla_X {F})( Y,\varphi^2 Z) =
 -(\nabla_X {F})(Y, QZ)
 -\eta(Z)[g((Q+h)X,Y)-\eta(X)\eta(Y)].$
Inserting this in \eqref{E-40}, gives
\begin{align}\label{E-41}
(\nabla_X {F})(\varphi Y, \varphi Z)
+ (\nabla_X {F})(Y, QZ)
= \eta(Y)g((Q+h)X,QZ) -\eta(Z)g((Q+h)X,Y) .
\end{align}
Using \eqref{dPhi}, we get
\begin{align*}
 &\quad (\nabla_X {F})(Y,Z)+ (\nabla_Y {F})(Z,X)+(\nabla_Z {F})(X,Y) +(\nabla_{\varphi X} {F})(\varphi Y,Z)+(\nabla_{\varphi Y}{F})(Z,\varphi X) \\
 & \quad+(\nabla_Z {F})(\varphi X, \varphi Y)
 +(\nabla_{\varphi X}{F})(Y, \varphi Z)
 +(\nabla_Y{F})(\varphi Z, \varphi X)
 +(\nabla_{\varphi Z}{F})(\varphi X,Y)\\
 & \quad -(\nabla_X {F})(\varphi Y, \varphi Z)
 -(\nabla_{\varphi Y} {F})(\varphi Z, X)-(\nabla_{ \varphi Z} {F})(X, \varphi Y) =0.
\end{align*}
Applying \eqref{E-39} and \eqref{E-41}, the above equality simplifies to
\begin{align} \label{E-41a}
 &2(\nabla_X {F})(Z,Y)+ 2 (\nabla_{\varphi X} {F})(Z,\varphi Y)  = 4\,{\eta}(Z) g(QX,QY)
 -[2\eta(Y)g((Q+h)X,QZ)\notag \\
 &+2\eta(X)\eta(Y) {\eta}(Z) -\eta(Y)g((Q+h)Z, \widetilde Q X)\notag
 +\eta(Z)g((Q+h)Y, \widetilde Q X)\\
 &-\eta(X)g((Q+h)Z, \widetilde Q Y)]
 +(\nabla_X {F})(Y,\widetilde QZ)-(\nabla_Y {F})(Z,\widetilde QX)-(\nabla_Z {F})(X,\widetilde QY).
\end{align}
Observe that using \eqref{E-nS-10}, and then the anti-symmetry of $\varphi$ and the equality $\nabla Q=\nabla \widetilde Q$, we have
\begin{align} \label{E-Phi}
 (\nabla_X {F})(Y,\widetilde QZ)=
 (\nabla_X {F})(\widetilde QY,Z)
 +\eta(Z)g((Q+h)X,\widetilde QY) +\eta(Y)g((Q+h)X,\widetilde QZ)
\end{align}
for all $X,Y,Z \in\mathfrak{X}_M$.
By the above and $d{F}=0$, the equation \eqref{E-41a} simplifies as
 \begin{align*}
 \notag
 & (\nabla_X {F})(Z,Y) + (\nabla_{\varphi X} {F})(Z,\varphi Y) =  2\,{\eta}(Z) g(QX,QY) - \big[\eta(Y)g((Q+h)X,QZ) \\
  & +\eta(X)\eta(Y) {\eta}(Z) +\tfrac12\,\eta(Z)\,g((Q+h)Y,\widetilde Q X)\big]
  +\tfrac12\big[(\nabla_X {F})(Y,\widetilde QZ)
  +(\nabla_{\widetilde Q X} {F})(Y,Z) \big] ,
 \end{align*}
that is equivalent to \eqref{E-04}.
\end{proof}

The subsequent result generalizes \cite[Lemma~7.4]{blair2010riemannian} of Olszak (1979), and is used in Proposition~\ref{prop-5.1}.

\begin{Proposition}\label{P-06}
For a w.c.m. manifold satisfying \eqref{Eq-2assump} and \eqref{E-nS-10}, the following is true:
\begin{align}\label{E-05b}
 \notag
 &
 g(R_{\xi, QX}Y, Z) -g(R_{\xi,X}\varphi Y, \varphi Z) +g(R_{\xi,\varphi X}Y, \varphi Z)
 +g(R_{\xi,\varphi X}\varphi Y, Z)  \\
 \notag
 &  = 2\,(\nabla_{\widetilde hX} {F})(Y,Z) +2\,{\eta}(Z)g((Q+h)X,QY)- 2\,{\eta}(Y)g((Q+h)X, QZ) \\
 \notag
 &  + \tfrac12\big[ 3(\nabla_{\widetilde Q\widetilde hX}{F})(Y,Z) +(\nabla_Y{F})(\widetilde hZ,\widetilde QX)
  +(\nabla_Z {F})(\widetilde QX,\widetilde hY) -(\nabla_{\widetilde QY} {F})(Z,\widetilde hX) \\
 &  -(\nabla_{\widetilde QY}{F})(\widetilde h Z,X) -(\nabla_{\widetilde QZ}{F})(\widetilde hX,Y) -(\nabla_{\widetilde QZ}{F})(X,\widetilde hY) \big].
\end{align}
\end{Proposition}

\begin{proof}
Using \eqref{E-xi-n},
proved based on the condition \eqref{Eq-2assump},
we find
\begin{align} \label{E-42}
 R_{Y,Z}\,\xi = -(\nabla_Y\,\varphi)Z
 + (\nabla_Z\,\varphi)Y
 - (\nabla_Y\,\varphi\widetilde h)Z
 + (\nabla_Z\,\varphi\widetilde h)Y.
\end{align}
Taking the scalar product of \eqref{E-42} with $X$ and simplifying using \eqref{dPhi}, we get the relation
\begin{align}\label{E-05}
 g(R_{\xi,X}Y ,Z)=-(\nabla_X {F})(Y,Z)-g(X,(\nabla_Y\,\varphi\widetilde h)Z)+ g(X, (\nabla_Z\,\varphi  \widetilde h)Y).
\end{align}
Using \eqref{E-05}, we have the following:
\begin{align}\label{E-42-tilde}
 g(R_{\xi,X}Y, Z) &-g(R_{\xi, X}\varphi Y, \varphi Z)
 + g(R_{\xi,\varphi X}Y, \varphi Z)
 +g(R_{\xi,\varphi X}\varphi Y, Z) \notag \\
 &= A(X,Y,Z) + \widetilde B(X,Y,Z) - \widetilde B(X,Z,Y) ,
\end{align}
where the operators $A$ and $\widetilde B$ are defined as follows:
\begin{align*}
 & A(X,Y,Z) = -(\nabla_X {F})(Y,Z) +(\nabla_X {F})(\varphi Y,\varphi Z) -(\nabla_{\varphi X} {F})(Y, \varphi Z)  -(\nabla_{\varphi X} {F})(\varphi Y,Z), \\
 & \widetilde B(X,Y,Z) = g(X, (\nabla_{\varphi Y}\varphi\widetilde h)\varphi Z) - g(X, (\nabla_Y\varphi\widetilde h)Z)
  - g(\varphi X, (\nabla_Y\varphi\widetilde h)\varphi Z) - g(\varphi X, (\nabla_{\varphi Y}\varphi\widetilde h)Z).
\end{align*}
Using \eqref{E-04}, \eqref{E-39}, \eqref{E-41} and \eqref{E-Phi}, the operator $A$ is simplified as
\begin{align}\label{eq-a1}
\notag
A(X,Y,Z) =& \,\big\{
 (\nabla_X {F})(Y,Z)+(\nabla_{X} {F})(fY,fZ)\big\}
+\big\{(\nabla_{fX} {F})(fY,Z)
-(\nabla_{fX} {F})(Y, fZ)\big\} \\
 =&\,2\,\eta(Z) g(QX, QY)
 -2\,\eta(Y) g(QX, QZ)
  +(\nabla_{\widetilde Q X} {F})(Y,Z) ,
\end{align}
and the operator $\widetilde B$ is read as
\begin{align*}
 &\widetilde B(X,Y,Z)
 = g(X, \varphi(\nabla_Y\,\widetilde Q\widetilde h)Z)
 -g( X,(\nabla_{Y}\,\varphi)\widetilde hZ)
+g(X, (\nabla_{\varphi Y}\,Q\widetilde h) Z)
+g(X,\widetilde h\varphi(\nabla_{\varphi Y}\,\varphi)Z) \\
&+g( X, \varphi(\nabla_{\varphi Y}\,\varphi) \widetilde h Z)
+g(Q X, \widetilde h(\nabla_Y\,\varphi) Z)
-g(Q X, (\nabla_{\varphi Y}\,\widetilde h) Z)
 +\eta(X)\,\eta((\nabla_{\varphi Y}\,\widetilde h) Z).
\end{align*}
 Using \eqref{E-nS-10} and \eqref{E-xi-n}, we get
\begin{align*}
 & g(X, (\nabla_{\varphi Y}\,Q \widetilde h) Z)
 = g(X, Q(\nabla_{\varphi Y}\,\widetilde h) Z)
 -\eta(X)g((Q-h)Y,\widetilde Q \widetilde hZ), \\
 & g(\varphi X, (\nabla_Y \widetilde Q \widetilde h)Z) = g(\varphi X, \widetilde Q( \nabla_Y \widetilde h)Z).
\end{align*}
These two equations allow us to simplify
$\widetilde B(X,Y,Z)$ as follows:
\begin{align*}
 &\widetilde B(X,Y,Z)
 =-g(X,(\nabla_{Y}\varphi)\widetilde hZ) -g(\widetilde Q \varphi X, (\nabla_Y\,\widetilde h)Z) +g(X, \widetilde h \varphi (\nabla_{\varphi Y}\,\varphi) Z) \\
 &\quad
 +g(Q X, \widetilde h(\nabla_Y \varphi) Z) +g(X,\varphi(\nabla_{\varphi Y}\,\varphi)\widetilde h Z) - \eta(X)g((Q{-}h)Y,\widetilde QhZ)
 +\eta(X)\,\eta((\nabla_{\varphi Y}\,\widetilde h) Z).
\end{align*}
In order to simplify $\widetilde B(X,Y,Z)$, we compute $\varphi(\nabla_{\varphi Y} \varphi)Z$ using \eqref{2.2}, \eqref{E-xi-n}, \eqref{E-04} and \eqref{E-nS-10b} as follows:
\begin{align}\label{e-phi}
 \varphi (\nabla_{\varphi Y} \varphi)Z
 & =(\nabla_{\varphi Y}\varphi^2)Z
 -(\nabla_{\varphi Y} \varphi)\varphi Z
 \notag \\
 & = -\big(g(QY,QZ)-\eta(Y)\eta(Z)\big){\xi} +2\,{\eta}(Z)\big(Q^2Y-\eta(Y)\xi\big)
 \notag \\
 & \quad -g(hY,QZ)\xi
 +(\nabla_Y\,\varphi)Z + P(Y,Z)
 +\tfrac{1}{2}\,g((Q+h)Z,\widetilde QY)\xi ,
\end{align}
where
$P(X,Y):=\frac{1}{2}\big[
\widetilde Q (\nabla_X\,\varphi)Y
 +(\nabla_{\widetilde QX}\,\varphi)Y \big]$.
Using
$\eta((\nabla_{\varphi Y} \widetilde h)Z)
=-g(\widetilde hZ, \nabla_{\varphi Y}\xi)
=g(Q\widetilde hZ, \widetilde hY-Y)$
and \eqref{e-phi},
we further simplify $\widetilde B(X,Y,Z)$ to
\begin{align}\label{Eq-b1}
\notag
 &\widetilde B(X,Y,Z) = 2\,g(\widetilde hX, (\nabla_{Y}\varphi) Z)
 +2\,{\eta}(Z)g( hX, QY) -2\,{\eta}(X)g(QY, hZ)+g(\widetilde{Q}X, \varphi (\nabla_Y \widetilde h)Z)  \\ & \quad
 +g(\widetilde{Q}X, \widetilde h(\nabla_{Y}\varphi) Z)
 +g(\widetilde h P(Y,Z)+P(Y, \widetilde hZ), X)
  + \tfrac{1}{2}\eta(X) g((Q+h)Y,\widetilde QZ)  .
\end{align}
We rewrite \eqref{E-42-tilde} using $d{F}=0$, \eqref{E-Phi}, \eqref{eq-a1}, and \eqref{Eq-b1} as follows:
\begin{align}\label{E-ABB}
 & A(X,Y,Z) + \widetilde B(X,Y,Z) - \widetilde B(X,Z,Y) \notag\\
 & \quad = 2\,(\nabla_{\widetilde hX}{F})(Y,Z)
 +2\,{\eta}(Z)g((Q+h)X,QY) -2\,{\eta}(Y)g((Q+h)X,QZ) \notag \\
 & \quad
 + (\nabla_{\widetilde QX} {F})(Y,Z) \notag + (\nabla_{\widetilde Q \widetilde hX} {F})(Y,Z)
 - g(\varphi\widetilde QX,\ (\nabla_Y\widetilde h)Z -(\nabla_Z\widetilde h)Y) \notag \\
 & \quad
 +g(\widetilde h P(Y,Z)+P(Y,\widetilde hZ), X)
 -g(\widetilde h P(Z,Y)+P(Z, \widetilde hY), X).
\end{align}
Using $d{F}=0$ and \eqref{E-Phi}, we simplify the last four terms in \eqref{E-ABB}:
\begin{align*}
2\,g(\widetilde h P(Y,Z)&+P(Y,\widetilde hZ), X)  -2\,g(\widetilde h P(Z,Y)+P(Z,\widetilde hY), X) \\
  & = (\nabla_{\widetilde Q \widetilde h X} {F})(Y,Z)
 + (\nabla_Y {F})(\widetilde QX,\widetilde hZ)
 -(\nabla_Z {F})(\widetilde QX,\widetilde hY)
 -(\nabla_{\widetilde QY} {F})(Z,\widetilde hX)
 \\
 & \quad  -(\nabla_{\widetilde QY} {F})(\widetilde hZ,X) +(\nabla_{\widetilde QZ} {F})(Y,\widetilde hX)
  +(\nabla_{\widetilde QZ}{F})(\widetilde hY,X).
\end{align*}
From \eqref{E-05} we have
\begin{align*}
 & g(\widetilde Q X, \varphi(\nabla_Z\widetilde h)Y -\varphi(\nabla_Y\widetilde h)Z)
 = g(\widetilde Q X, (\nabla_Z\,\varphi\widetilde h)Y -(\nabla_Y\,\varphi\widetilde h)Z)
 -g(\widetilde Q X, (\nabla_Z\,\varphi)\widetilde h Y -(\nabla_Y\,\varphi)\widetilde h Z) \\
 & = g(R_{\xi,\widetilde Q X}Y, Z)
 + (\nabla_{\widetilde QX} {F})(Y,Z)
 -(\nabla_Z{F})(\widetilde Q X,\widetilde h Y)
 +(\nabla_Y{F})(\widetilde Q X,\widetilde h Z).
 \end{align*}
Finally, using the two expressions above, we simplify \eqref{E-ABB} to the following:
\begin{align} \label{E-CURV}
 &A(X,Y,Z) + \widetilde B(X,Y,Z) - \widetilde B(X,Z,Y)\notag \\
 \notag
  =&\,2\,(\nabla_{\widetilde hX} {F})(Y,Z)
 +2\,{\eta}(Z)g((Q+h)X,QY) -2\,{\eta}(Y)g((Q+h)X,QZ) \notag \\
 \notag
 &+ \tfrac{1}{2}\big[ 3(\nabla_{\widetilde Q\widetilde hX}{F})(Y,Z)
 -(\nabla_{\widetilde QY}{F})(Z,\widetilde hX)
 -(\nabla_{\widetilde QY}{F})(\widetilde h Z, X) +(\nabla_Y {F})(\widetilde hZ, \widetilde QX)\\
 &+ (\nabla_Z {F})(\widetilde QX,\widetilde hY)
 -(\nabla_{\widetilde QZ}{F})(X,\widetilde hY) -(\nabla_{\widetilde QZ}{F})(\widetilde hX,Y) \big]
 -g(R_{\xi,\widetilde Q X}Y, Z).
\end{align}
The required equality \eqref{E-05b} follows from \eqref{E-CURV}.
\end{proof}

The following statement is used in the proofs of the main~theorems.

\begin{Proposition}\label{prop-5.1}
Let a w.c.m. manifold
\(M^{2n+1}(\varphi,Q,\xi,\eta,g)\) satisfy conditions
\eqref{E-k-mu-nullity} and \eqref{Eq-2assump}.
Then $\kappa\le1$.
(i)~If $\kappa<1$ then
$\ker\widetilde h=\{\xi\}$ and $\ker\eta={\cal D}_+ \oplus {\cal D}_-$,
where ${\cal D}_\pm$ are $n$-dimensional eigen-distributions of the eigenvalues $\pm\sqrt{1-\kappa}$
of~$\widetilde h$.
(ii)~If $\kappa<1$ and, in addition, \eqref{E-nS-10} is valid,
 then
\begin{align}\label{Eq-phi03}
 (\nabla_X {F})(Y,Z) &= 0
\quad (X,Z\in{\mathcal D}^-,\ Y\in {\mathcal D}^+),\\
 \label{Eq-phi02}
 (\nabla_X {F})(Y,Z) &=0 \quad
 (X,Y\in {\cal D}_+,\ Z \in\ker\eta), \\
 \label{Eq-phi001}
  (\nabla_X {F})(Y,Z) &=0 \quad
 (X,Y,Z\in {\cal D}_-).
 \end{align}
\end{Proposition}

\begin{proof}
$(i)$ From the $(\kappa,\mu)$-nullity condition
\eqref{E-k-mu-nullity}, we get
 $QR_{\,\xi, X}\, \xi=\kappa\,Q\varphi^2X-\mu\, QhX$
 and  $\varphi R_{\,\xi,\,\varphi X}\,\xi =
-\kappa\,Q\,\varphi^2 X -\mu\,QhX$.
Combining these two equations, we acquire
\begin{align}\label{E-31a}
 QR_{\,\xi, X}\,\xi
 -\varphi R_{\,\xi,\,\varphi X}\,\xi
 = 2\kappa\,Q\varphi^2 X.
\end{align}
Comparing \eqref{E-31a} with \eqref{E-37}
where \eqref{Eq-2assump} is used, we get
$h^2X =(1-\kappa)\big\{Q^2X -\eta(X) \xi\big\}$,
hence
$\widetilde h^{\,2}X=(1-\kappa)X\ (X\in\ker\eta)$.

Since $\widetilde h=Q^{-1}h$ is self-adjoint,
$\widetilde h^{\,2}$ is non-negative definite; hence
$\kappa\le 1$ holds. Since $\widetilde{h}\,\xi=0$ is true, we have $\ker\widetilde h=\{\xi\}$ and $\ker\eta={\cal D}_+ \oplus {\cal D}_-$,
where ${\cal D}_\pm+$ are the eigen-distributions for the eigenvalues $\pm\sqrt{1-\kappa}$.
Using $\varphi h = -h \varphi$ and $\varphi Q = Q \varphi$, we obtain
$\widetilde{h}(\varphi X)=-\sqrt{1-\kappa}\,\varphi X\ (X\in{\cal D}_+)$, which ensures
$\varphi X\in{\cal D}_-$. Thus, $\varphi$ maps ${\cal D}_+$ onto ${\cal D}_-$ and vice versa, so $\dim {\cal D}_+ = \dim {\cal D}_- = n$. Observe that $g(\varphi X, \varphi Y)=0$ for $X \in {\cal D}_+$ and $Y \in {\cal D}_-$, which gives $g(QX,Y)=0$. Thus,
$Q({\cal D}_+)={\cal D}_+$ and $Q({\cal D}_-)={\cal D}_-$,
similarly, $\widetilde Q({\cal D}_+)={\cal D}_+$ and $\widetilde Q({\cal D}_-)={\cal D}_-$.

\smallskip

$(ii)$ From \eqref{E-05b} (where \eqref{E-nS-10} is used) with $X,Y,Z\in\ker\eta$ and using \eqref{E-k-mu-nullity}, we obtain the following:
\begin{align}\label{Eq-phi01}
  & 4(\nabla_{\widetilde hX} {F})(Y,Z) = (\nabla_{\widetilde QY} {F})(Z,\widetilde hX) {+}(\nabla_{\widetilde QY}{F})(\widetilde h Z,X)
   {+} (\nabla_{\widetilde QZ} {F})(\widetilde hX,Y) {+} (\nabla_{\widetilde QZ} {F})(X,\widetilde hY) \notag \\
 &-3(\nabla_{\widetilde Q\widetilde hX} {F})(Y,Z)
 -(\nabla_Y {F})(\widetilde hZ, \widetilde QX)  -(\nabla_Z {F})(\widetilde QX,\widetilde hY) \quad (X,Y,Z \in \ker\eta).
\end{align}
Taking $X,Z\in{\mathcal D}^-$ and $Y\in{\mathcal D}^+$ in \eqref{Eq-phi01}, and applying \eqref{dPhi} along with the fact $\widetilde h V= \pm \sqrt{1-\kappa}\,V$ depending on whether $V\in{\cal D}_-$ or $V\in{\cal D}_+$, we find
\begin{align}\label{Eq-N}
 2\, (\nabla_X {F})(Y,Z) =-(\nabla_{\widetilde Q X} {F})(Y,Z) -(\nabla_{X}{F}) (\widetilde QY,Z).
\end{align}
Cyclically changing $X,Y,Z$ in \eqref{Eq-N}
and applying \eqref{dPhi}, we acquire
\begin{align*}
 0&= -2\,d{F}(X,Y,Z)= - 2\, [(\nabla_X {F})(Y,Z)+(\nabla_Y {F})(Z,X)+ (\nabla_Z {F})(X,Y)] \\
 & = (\nabla_{\widetilde Q X} {F})(Y,Z)
 + (\nabla_{X}{F}) (\widetilde QY,Z)
 + (\nabla_{\widetilde Q Y} {F})(Z,X)
 + (\nabla_{Y}{F}) (\widetilde QZ,X) \\
 & \quad + (\nabla_{\widetilde Q Z} {F})(X,Y)
 + (\nabla_{Z}{F}) (\widetilde QX,Y)
 = (\nabla_{\widetilde Q X} {F})(Y,Z)
 - (\nabla_{X}{F}) (\widetilde QY,Z).
\end{align*}
Using this in \eqref{Eq-N}, gives $(\nabla_X\,{F})(QY,Z)=0$, which implies \eqref{Eq-phi03} as $Q$ is a bijection on ${\cal D}_+$.
Taking $X,Y,Z\in {\mathcal D}^+ $(or in ${\cal D}_-$)  in \eqref{Eq-phi01}, we find
\begin{align*}
 4(\nabla_{X}{F})(Y,Z) &= -3(\nabla_{\widetilde QX}{F})(Y,Z)
 +2  (\nabla_{\widetilde QY}{F})(Z,X) +2(\nabla_{\widetilde QZ}{F})(X,Y) \\
 &\quad -(\nabla_{Z}{F})(\widetilde Q X,Y) -(\nabla_{Y}{F})(Z,\widetilde Q X).
\end{align*}
Using this and \eqref{dPhi}, we get $(\nabla_{X}{F})(QY,Z)=0$, which indicates that $(\nabla_{X}{F})(Y,Z)=0$ for $X,Y,Z\in {\mathcal D}^+$(or in ${\cal D}_-$), that is, \eqref{Eq-phi02} is true. In a similar manner, we get \eqref{Eq-phi001}.
\end{proof}

\section*{Conclusion}

This work demonstrates that several fundamental rigidity phenomena of contact metric geometry continue to hold in the broader setting of weak contact metric manifolds. In particular, by employing the partial Ricci flow, we obtain convergence of w.c.m. structures to Sasakian structures under suitable assumptions, find conditions under which a w.c.m.~mani\-fold admits a bi-Legendrian structure, and establish corresponding classification results.
This work contributes to research aimed at characterizing contact structures via curvature conditions, as well as to the broader program of studying the evolution of geometric structures under curvature-driven dynamics -- a topic of increasing relevance in geometric analysis and mathematical physics.
We expect that the weak contact metric structures studied here will find further applications in CR-manifolds, complex contact manifolds \cite{blair2010riemannian}, and non-symmetric Riemannian geometry \cite{RZ-1}.
It would be interesting to investigate whether these rigidity and convergence properties remain valid under weaker hypotheses and to explore their implications in models of
non-symmetric theories of gravity.

\end{document}